\documentclass[11pt, report]{amsart}

\usepackage{ifthen}
\usepackage[T1]{fontenc}
\usepackage[utf8]{inputenc}
\usepackage[all]{xy}
\usepackage{bbold}
\usepackage{url}
\usepackage{color}

\usepackage
[pdftex,
colorlinks, 
pdfpagemode=false, 
bookmarksnumbered, 
bookmarksopen, 
pdfpagelabels=false,
pdfstartview=FitH, 
linkcolor=blue, 
citecolor=green, 
urlcolor=blue, 
filecolor=red, %
backref, 
]{hyperref}

\usepackage{leftidx}

\newtheorem{anyprop}{Anyprop}[section]

\newtheorem{theorem}[anyprop]{Theorem}
\newtheorem{lemma}[anyprop]{Lemma}
\newtheorem{prop}[anyprop]{Proposition}
\newtheorem{corollary}[anyprop]{Corollary}
\newtheorem{condition}[anyprop]{(FC)}
\newtheorem*{resulta}{\textbf{Theorem A}}  
\newtheorem*{resultb}{\textbf{Theorem B}}  
\newtheorem*{resultc}{\textbf{Corollary C}}
\newtheorem*{resultd}{\textbf{Theorem D}}
\newtheorem*{resulte}{\textbf{Proposition E}}

\theoremstyle{definition}

\newtheorem{definition}[anyprop]{Definition}

\newtheorem{example}[anyprop]{Example}
\newtheorem{counterexample}[anyprop]{Counterexample}

\newtheorem{construction}[anyprop]{Construction}

\newtheorem{remark}[anyprop]{Remark}

\newtheorem*{acknowledgement}{Acknowledgement}

\theoremstyle{remark}

\numberwithin{equation}{section}

\usepackage{amscd}
\usepackage{amssymb}

\usepackage[none]{optional}
\newcommand{\isoto}{\overset{\sim}{\to}}
\newcommand{\isotomap}[1]{\underset{\sim}{\overset{#1}{\to}}}

\newcommand{\isoeq}{\cong}
\newcommand{\lie}{\mathrm{Lie}~}

\newcommand{\im}{\mathrm{Im}~}
\newcommand{\nato}{\mathbb{N}_0}
\newcommand{\nat}{\mathbb{N}}
\newcommand{\Dim}[1]{\underset{#1}{\mathrm{dim}}~}
\newcommand{\lgth}{\mathrm{lg}~}
\newcommand{\rank}[1]{\underset{#1}{\mathrm{rank}}~}

\newcommand{\zz}{\mathbb{Z}}
\newcommand{\zd}{\mathcal{Z}}
\newcommand{\ez}{E_{\mathcal{Z}}}
\newcommand{\ezchi}{E_{G, \chi}}

\newcommand{\fq}{\mathbb{F}_{q}}
\newcommand{\fp}{\mathbb{F}_{p}}
\newcommand{\qp}{\mathbb{Q}_{p}}
\newcommand{\os}{\mathcal{O}_{S}}
\newcommand{\ofun}[1]{\mathcal{O}_{#1}}
\newcommand{\Spec}{\mathrm{Spec}~}
\newcommand{\neured}[1]{{#1}_{\mathrm{red}}^{0}}
\newcommand{\red}[1]{ ({#1})_{\mathrm{red}}}
\newcommand{\intt}[1]{\mathrm{int}\left ( #1 \right )}
\newcommand{\multgrp}{\mathbb{G}_{m,k}}
\newcommand{\GLn}{\mathrm{GL}_{n,k}}
\newcommand{\GL}[1]{\mathrm{GL}_{#1}}
\newcommand{\aff}[1]{\mathbb{A}_{#1}}

\newcommand{\ezn}{E_{\mathcal{Z}_{n}}}
\newcommand{\urad}{\mathcal{R}_{u}}

\newcommand{\conj}[2]{\leftidx{^{#1}}{#2}{}}
\newcommand{\fiber}[3]{#1 \underset{#3}{\times} #2}
\newcommand{\ten}[3]{#1 \underset{#3}{\otimes} #2}
\newcommand{\orb}{\mathfrak{o}} 
\newcommand{\hdr}[2]{H_{\mathrm{dR}}^{#1}(#2)}

\newcommand{\kd}[3]{\Omega^{#1}_{#2/#3}}
\newcommand{\hdss}[2]{\leftidx{_{H}}{E_{1}^{#1#2}}{}}
\newcommand{\hdssr}[3]{\leftidx{_{H}}{E_{#1}^{#2#3}}{}}
\newcommand{\hdssinf}[2]{\leftidx{_{H}}{E_{\infty}^{#1#2}}{}}
\newcommand{\conjss}[2]{\leftidx{_{conj}}{E_{2}^{#1#2}}{}}
\newcommand{\conjssinf}[2]{ \leftidx{_{conj}}{E_{\infty}^{#1#2}}{}}
 
\newcommand{\Orb}{\mathfrak{O}}

\newcommand{\Fr}{\mathrm{Frob}}
\newcommand{\codim}[2]{\mathrm{codim}(#1,#2)}

\newcommand{\stab}[2]{\mathrm{Stab}_{#1}(#2)}

\newcommand{\pr}[2]{\pi_{#1}(#2)}
\newcommand{\qstack}[2]{\left[#1 \backslash #2 \right]}

\newcommand{\btsm}{\mathcal{BT}^{n,d}_{m}}
\newcommand{\Frq}[1]{#1^{(q)}}
\newcommand{\Frqb}[1]{\left(#1\right)^{(q)}}
\newcommand{\Frp}[1]{#1^{(p)}}
\newcommand{\Frpb}[1]{\left(#1\right)^{(p)}}
\newcommand{\gzipdata}[2]{(#1, #1_{+}, #1_{-}, #2)}
\newcommand{\gzipcat}{G\text{-}\mathrm{Zip}_{k}^{\chi}}
\newcommand{\glnzipcat}{\mathrm{GL}_{n}\text{-}\mathrm{Zip}_{k}^{\chi}}

\newcommand{\weili}{{}^IW} 
\newcommand{\iweilj}{{}^IW^{J}} 
\newcommand{\siw}{S_{\underline{I}}^{[w]}}
\newcommand{\siwcl}[1]{S_{\underline{I}}^{#1}}
\newcommand{\kscheme}{(\mathrm{\textbf{Sch}}/k)}
\newcommand{\catsch}{(\mathrm{\textbf{Sch}})}
\newcommand{\fzipdatanew}[2]{(#1,C^{\bullet},D_{\bullet}, #2_{\bullet})}
\newcommand{\fzipnew}[1]{\underline{#1}}
\newcommand{\tatezip}[1]{\underline{\mathbb{1}}(#1)}
\newcommand{\fzipdata}[2]{(\mathcal{#1},C^{\bullet},D_{\bullet}, #2_{\bullet})}
\newcommand{\gra}[1]{\mathrm{gr}^{D}_{#1}}
\newcommand{\grd}[1]{\mathrm{gr}_{C}^{#1}}
\newcommand{\fzip}[1]{\underline{\mathcal{#1}}}
\newcommand{\fzipcat}{F\text{-}\mathrm{Zip}}
\newcommand{\fzipcatn}{F\text{-}\mathrm{Zip}_{k}^{\underline{n}}}
\newcommand{\aut}[1]{\mathrm{Aut}\left( #1 \right)}
\newcommand{\autm}[1]{\underline{\mathrm{Aut}}\left( #1 \right)}
\newcommand{\isom}[2]{\underline{\mathrm{Iso}}\left( #1, #2 \right)}
\newcommand{\verein}[1]{\underset{#1}{\bigcup}}
\newcommand{\osum}[1]{\underset{#1}{\bigoplus}}
\newcommand{\nsum}[1]{\underset{#1}{\sum}}
\newcommand{\diedonne}{\mathbb{D}}
\newcommand{\mx}{\mathcal{M}(X)}

\newcommand{\grpk}{\mathcal{K}}

\newcommand{\grpi}{\mathcal{I}}
\newcommand{\gmodcat}[1]{{}_{\mathcal{#1}} \mathcal{M}}
\newcommand{\rmodart}[1]{#1\text{-}\mathrm{MOD_{fl}}}
\newcommand{\markme}[2]{\opt{nb}{\textcolor{red}{\textbf{#1}} \marginpar{\begin{flushleft}\textcolor{red}{\textbf{\textit{#2}}}\end{flushleft}}}} 
\newcommand{\basechg}[2]{(#1)_{#2}}
\begin{document} 

\definecolor{drkgreen}{rgb}{0,0.3,0}
\definecolor{drkblue}{rgb}{0.1,0,0.4}

\hypersetup{linkcolor=drkblue, citecolor=drkgreen}
\title[Purity of $G$-zips]
{Purity of $G$-zips}

\author[Yaroslav Yatsyshyn]{Yaroslav Yatsyshyn}
\thanks{This work was partially supported by the German Research Foundation (DFG)}


\email{yatsyshyn@gmx.de}




%
%
%
%
%
%
%
%
%
%
%
%

\maketitle

\makeatletter
\@setabstract
\makeatother
\renewcommand{\abstractname}{Abstract}
\begin{abstract}
Let $k$ be a perfect field of characteristic $p>0$, and $S$ an scheme over $k$.
 An $F$-zip is basically  a locally free $\os$-module of finite rank endowed with two filtration and an Frobenius-linear isomorphism between their graded pieces.
  The natural generalization of this notion for a reductive algebraic group  $G/k$ is an ``$F$-zip with $G$-structure'', a so-called $G$-zip introduced in \cite{fzips}.
A $G$-zip $\underline{I}$ over $S$ yields the stratification of the base scheme $S= \verein{w} \siw$ in loci, where  $\underline{I}$ has  locally a constant isomorphism class for the fppf topology.
We show that $\siw \hookrightarrow S$ are affine and give a number of geometric applications of this purity result.
 \end{abstract}
\makeatletter
\@setabstract
\makeatother

\makeatletter
\def\Links{\tagsleft@true}\def\Rechts{\tagsleft@false}
\makeatother 
\Rechts
\section*{Introduction}
\subsection*{Background and motivation}
\sloppy
Let $k=\fp$ for the sake of simplicity unless stated otherwise.

Giving a short historical account of purity problems in the algebraic and arithmetic geometry one should mention the  Purity Theorem (2000) of de Jong-Oort \cite{oort}:
 \begin{theorem} \label{jo}
 Let $S$ be an integral, excellent scheme in characteristic $p$. Let $\mathfrak{X} \to S$ be a Barsotti-Tate group over $S$. 
Further let $U \subset S$ be the largest (open dense) \footnote{these properties are automatically satisfied for a such set, see \cite[Thm. 2.3.1, p. 143]{katzsl}} set
 on which the Newton polygon is constant. Then, either $U=S$, or  $S-U$ has codimension one in $S$. 
 \end{theorem}

Let us remark that  one could require some other regularity/finiteness conditions instead of ``excellence'' of $S$.

This kind of result is referred by  A. Vasiu in \cite{vasiu2} as ``the weaker variant of purity''. 
 In fact, he shows a stronger version of the above theorem which implies the de Jong-Oort's result by applying the standard Hartogs-like yoga:

 \begin{theorem}
 Let $\mathfrak{X}$, $U \subset S$ are as in Theorem \ref{jo}.
 Then the open inclusion $U \hookrightarrow S$ is affine. 
 \end{theorem}
Afterwards F. Oort gave an alternative proof of the above theorem in his conference talk (see \cite{repurity}) similar in flavor to that of  A. Vasiu.
 
The authors of \cite{vasiu} consider another purity problem for Barsotti-Tate groups: 
Pick $m \in \nat$, and let $S$ an arbitrary scheme over $k$.
Let $\mathfrak{X}_m$ be an $m$-truncated Barsotti-Tate group over $S$. 
Further let  $S_{X'}^m$ be the subscheme of $S$ that describes the locus where the  $\mathfrak{X}_m$ is locally 
for the fppf topology isomorphic to $\fiber{X'}{S}{k}$, where $X'$ is an $m$-truncated  Barsotti-Tate group over $k$.
As shown in \textit{loc.cit.}  the assertion $S_{X'}^m \hookrightarrow S$ affine holds for all primes $p \geq 5$, and 
 under some strong conditions on $X'$ it holds also  
for $p \in \{2,3\}$.  One should mention  that the core of the proof is based on the case $m=1$; this case readily implies the case $m>1$. 
For $m=1$ this purity result is equivalent to purity for a special class of $F$-zips, see below for an informal introduction to them. 
 
Another motivation for this work comes from  the fact that some data of geometrical origin, e.g., de Rham cohomology groups of certain projective varieties,
 has a structure of a so-called $F$-zip with maybe some additional structures.
The notion of an $F$-zip was first introduced in \cite{ben}.  Its authors B. Moonen and T. Wedhorn studied the de Rham cohomology $\hdr{n}{X/S}$ of a smooth proper scheme $f\colon X  \to S$. 
They showed  that under assumption of the so-called condition (D) which  says:  i) the higher direct images $R^{a}f_{*} \kd{b}{X}{S}$ for $a,b \in \nato$ are locally free $\os$- modules of  finite rank, and 
ii) the Hodge-de Rham spectral sequence degenerates at $E_1$, follows
that $M:=\hdr{n}{X/S}$ carries a structure of an $F$-zip, i.e.  $M$ is endowed with  two filtrations (``Hodge'' and ``conjugate'' filtration), and there is a Frobenius-linear morphism between their graded pieces 
induced by the Cartier isomorphism.
For a general reductive algebraic group $G$,  R. Pink, T. Wedhorn,  P. Ziegler  defined in  \cite{fzips}  the notion of an $F$-zip with $G$-structure, called a \textit{$G$-zip} (see Definition \ref{gzps}). 
These additional structures arise naturally:  For instance assume that $f\colon X \to S$ is of pure dimension $d$ with geometrically connected fibers satisfying  condition (D). Then the cup pairing on
the ``middle'' de Rham cohomology group $\hdr{d}{X/S}$ gives rise to a symplectic (resp. a symmetric) pairing for $d$ odd (resp. even).
In this case one obtains a $G$-zip, where 
$G = \mathrm{CSp}_{h,k}$ (resp. $G = \mathrm{CO}_{h,k}$), which
is the  group of symplectic simultudes (resp. of orthogonal simultudes) for $h = \rank{\os}\hdr{d}{X/S}$, see \cite[§8]{fzips}. Another example are $F$-zips with additional structures 
associated to abelian varieties with certain extra data determined by a Shimura PEL-datum, see \cite{eva}. 
   
A $G$-zip over $S$ yields a stratification of a base scheme $S$ in similar fashion as explained above in case of an $m$-truncated Barsotti-Tate group over $S$. 
In its turn, giving an $F$-zip of rank $n$ is equivalent to giving an $\GLn$-zip.

In case of  $S=\Spec k$ specifying these two filtrations for an $F$-zip is equivalent to giving two opposite parabolic subgroups of $\GLn$, and a Frobenius-linear map  between their Levi-factors. 
The generalization thereof leads to a concept of \textit{algebraic zip datum} introduced in \cite{paul}, which is a quadruple $\mathcal{Z}=(G,P,P',\varphi)$, where $G$ is a reductive algebraic group over $k$,
$P$ and $P'$ are parabolic subgroups with unipotent radicals $R_u P$ resp.  $R_u P'$, and an isogeny $\varphi\colon P/R_u P \to P'/R_u P'$. 

One could ask whether $(G, P, P', \varphi)$ and $(G, P, P', \psi)$ define the same algebraic zip datum up to a change of basis. 
To tackle this problem one defines an action of the associated zip group $\ez =\{ (p',p) \in P' \times P: \varphi([p'])= [p]\}$ on $G$ given by  
$((p',p),g) \mapsto p'gp^{-1} \blacklozenge $. 

The elements $g$ and $g' \in G$ lie on the same orbit whenever they correspond the same $\varphi$ up to a change of basis. 

Let us remark that the notion of (non-connected) algebraic zip datum considered in \cite{fzips} has a more general setting as above 
with $G$, $P$ and $P'$  playing a r\^{o}le of  the neutral components of some, in general non-connected, algebraic groups. 
But the purity problem considered here can be reduced to the case of connected algebraic zip datum, see also Remark \ref{czdred}, hence
we limit ourself to study the connected version.  

A crucial  r\^{o}le in \cite{fzips} as well as in this  paper plays the algebraic stack $\qstack{\ez}{G}$, which is the quotient stack  with respect to the above action.

A $G$-zip is an object over $S$ that locally for \'{e}tale topology looks like a family of algebraic zip data parametrized by some section $g \in G(S)$ (see \cite[Lemma 3.5.]{fzips}). 
It turns out that their classifying stack is isomorphic to $\qstack{\ez}{G}$.

\subsection*{Results}
Let $k$ here be  a perfect field containing $\fp$, and $S$ be a $k$-scheme. 
 
Basically in this paper, we prove the following purity result and give several applications:

\begin{resulta} \label{athm}
Let $k$ be an algebraically closed field.  Suppose that $G$ contains  a finite number of $\ez$-orbits with respect to the action $\blacklozenge$  on it.
Then  $\ez$ acts on $G$ with affine orbits.
\end{resulta}
 
The above Theorem  implies the following easy but important corollary:

\begin{resultb} \label{corB}
 Let $\underline{I}$ be a $G$-zip of over $S$. 
 $\underline{I}$ yields the finite decomposition $\label{decomB} S= \verein{w} \siw$ in loci, where $\underline{I}$ has  locally a constant isomorphism class for the fppf topology.
Then $\siw \overset{\jmath}{\hookrightarrow} S$ is affine.
\end{resultb}
Check section \ref{pure_strat} for more details about the index set of the above decomposition.

In its turn, the above strong purity result implies the following weak purity result:
\begin{resultc} 
Suppose that  $S$ is a locally noetherian $k$-scheme, $Z$  a closed subscheme of $S$ of codimension $\geq 2$, which contains no embedded components of $S$
that the restriction of $\underline{I}$ to $S \setminus Z$ is fppf locally constant, 
then $\underline{I}$ is fppf locally constant.
\end{resultc}  
 
Next we give a short new proof of the main result in \cite{vasiu} about the purity of the stratification of a basis scheme $S$
based on the local isomorphism class of $\mathfrak{X}_m$ discarding all restrictions in characteristics $2$ and $3$:

\begin{resultd}
In the notation  of the previous section holds:
 The inclusion $S_{X'}^{m} \hookrightarrow S$ is affine.
\end{resultd}

Let now $X$ be a smooth proper scheme over $S$.
At the very end of the paper  we give some sufficient conditions and examples when de Rham cohomology  $\hdr{n}{X/S}$ carries the structure of $F$-zip making in particular
the purity result applicable in this case.

\begin{resulte}
Let $f\colon X \to S$ be a smooth proper morphism of schemes.
Suppose that there is a lift of $X$ in zero characteristic (see Definition \ref{liftchar0}), $\tilde{f}\colon \tilde{X} \to \tilde{S}$  such that 
$\tilde{X}$ and $\tilde{S}$ are locally noetherian schemes, $\tilde{f}$ is proper and smooth, and $\tilde{S}$ reduced. 

 Further assume the Hodge numbers $\tilde{s} \mapsto \Dim{\kappa(\tilde{s})} H^{b}(\tilde{X}_{\tilde{s}}, \kd{a}{\tilde{X}_{\tilde{s}}}{\kappa(\tilde{s})})$
are locally constant on $\tilde{S}$ for all $a,b \in \nato$.

Then $f$ satisfies condition (D).
\end{resulte}
We also give examples of application of the last proposition.

\subsection*{Content}

This paper is organized as follows. Section 1 contains a short recollection of basic facts about algebraic zip datum, the associated
quotient stack, and $F$- and $G$-zips presented in \cite{fzips} and \cite{paul}.

Section 2  gives an insight in the geometry of the orbits in Theorem A, and  culminates in its proof.

In section 3 will be explained how  Theorem A implies purity results for the strata of Theorem B, and 
the week purity result Corollary C.

Section 4 outlines some applications of the purity results: In subsections 4.1 and 4.2 we concern us with the purity result of \cite{vasiu}, see Theorem D.   

Section 5 focuses on the  de Rham cohomology  $\hdr{n}{X/S}$ of a proper smooth variety over $S$,  and on conditions upon which it carries an $F$-zip structure.
It discusses also some examples. 

\begin{acknowledgement}
 I would like to express my deep gratitude to my PhD thesis advisor Torsten  Wedhorn. This work would not have been possible without his encouragement and support.
 He also owns my special thanks for the careful reading of this paper, his comments and corrections, 
 and for the patience treating my knowledge gaps in algebraic/arithmetic geometry and lack of expertise therein.

I am also grateful to Ralf Kasprowitz and Eike Lau for many helpful discussions and suggestions.
\end{acknowledgement}

\newpage

\section{Preliminaries: General notation and basic facts}

\subsection*{Algebraic zip datum}
Let $k$ be a field extension of a finite field $\fq$ of order $q$, which is a perfect field, and let $S$ be a scheme over $k$. 
We denote by $G$  a (connected) reductive quasi-split algebraic group over the field $k$, fix $T \subset G$  a maximal torus and
 $T \subset B \subset  G$ a Borel subgroup. Further let $P, P' \subset  G$  be  parabolic subgroups such that $B \subset P$ and $\conj{g_0}{B} \subset P'$ for some fixed element 
$g_0 \in G$.

 Denote by  $U$ and $U'$ the unipotent radicals of $P$ resp. $P'$ and by $L$ and $L'$ their unique Levi-factors verifying $T \subset L$ and  $\conj{g_0}{T} \subset L'$.
In this way, we obtain two canonical projections $\pi_{L}\colon P \to L$,   $\pi_{L'}\colon P' \to L'$. 

Furthermore, we restrict our attention to such pairs $(P',P)$, such there is  an isogeny  $\varphi\colon L' \to L$ satisfying the  constraints $\varphi(\conj{g_0}{B} \cap L')=B \cap L$  
and   $\varphi(\conj{g_0}{T})=T$. 

We recall the following central definition introduced in \cite{paul}:
\begin{definition}
1) A \textit{connected algebraic zip datum} \footnote{this definition was originally made in the case  of algebraically closed field $k$}  $\mathcal{Z}$ 
is a quadruple $\mathcal{Z}=(G,P,P',\varphi)$ as above.

2) The linear algebraic group $\ez$ over $k$  given by \begin{equation} \label{zipgrp} \ez =\{ (p',p) \in P' \times P: \varphi(\pr{L'}{p'})= \pr{L}{p} \} \end{equation}
is called  \textit{zip group associated  to $\mathcal{Z}$}. 
\end{definition}

The group $\ez$ acts on $G$ by: \begin{equation} \label{zipaction} ((p',p),g) \mapsto p'gp^{-1} \text{ for } (p,p') \in \ez, g \in G \end{equation}
Or, more explicitly, writing $P'=U' \rtimes L'= U' \cdot L'$ and $P= U \rtimes L= U \cdot L$, $p'=u'l'$, $p=ul$,  this action becomes:
 
\[((p',p),g) \mapsto u'l'g \varphi(l')^{-1}u^{-1}.\]

Moreover, we impose the following additional condition:

\begin{condition} \label{fc}
For an algebraic closure $\bar{k}$ of $k$  there is only a finite number of  $\ez(\bar{k})$-orbits of $G(\bar{k})$.
\end{condition}

We will see in the section \ref{aff}  that the condition (FC) is in particular fulfilled if $\lie \varphi=0$,  
but in fact the latter condition is too strong.

Throughout this paper we consider the algebraic quotient stack  $\qstack{\ez}{G}$. 

The geometric situation described in \cite{fzips} leads to some special kind of algebraic zip datum associated to a cocharacter $\chi\colon \multgrp \to G$.
We assume that the reductive algebraic group $G$ is defined over $\fq$ i.e. $G=G'_{k}$, where $G'$ is a reductive algebraic group over $\fq$. Let $L$ be the centralizer of $\chi$ in $G$.  
Then, there are two opposite parabolic subgroups $P_{\pm}=L \ltimes U_{\pm}$ with the common Levi factor $L$ and the unipotent radicals  $U_{\pm}$, where the 
Lie algebras $\mathfrak{u}_{\pm}$ are directs sums of positive resp. negative weight spaces in the Lie algebra $\mathfrak{g}$ under $\mathrm{Ad} \circ \chi$. 

We denote by  $ \Frqb{.}$ a pullback of a scheme or a sheaf under $q^{\mathrm{th}}$- power absolute Frobenius map $S \to S$ resp. $k \to k$.  

Clearly, we have  $\Frq{G}=G$. 

\begin{lemma} \label{choice}
Let $G$ be a reductive algebraic group over $k$ defined over $\fq$. Furthermore, let $P$ be a parabolic subgroup of $G$ and $L \subset P$ be a Levi subgroup. 
There exist a maximal torus $T$, a Borel subgroup $B$ of $G$ both already defined over $\fq$, and $\bar{g} \in G(k)$  such that $T \subset \conj{\bar{g}}{L}$ 
and $T \subset B \subset \conj{\bar{g}}{P}$.
\end{lemma}
%
%

\begin{proof}
By the assumption, $G$ is a quasi-split algebraic group, thus we can choose a torus  $T$ and a Borel subgroup $B \supset T$ defined over $\fq$.

By  \cite[Expos\'{e} XXVI, Lemme 3.8.]{sga33} there is the parabolic subgroup $P'$ such that $B \subset P'$ , and $P'$ is of the same type as $P$.
By  Proposition 1.6 \textit{loc.cit.} there is the unique Levi subgroup $L'$ of $P'$ such that $T \subset L'$.
Then the assertion of the lemma is a direct consequence of Corollaire 5.5.(iv) \textit{loc.cit.}.
 \end{proof}

A new zip datum $(G, \conj{\bar{g}}{P}, \conj{\bar{g}}{P'}, \intt{\bar{g}} \circ \varphi \circ \intt{\bar{g}^{-1}})$ for a $\bar{g} \in G(k)$
is obviously  equivalent to the original one, so we assume  by the previous lemma that there are a maximal torus $T \subset L$ and a Borel subgroup $P \supset B \supset T$ already defined over $\fq$. 

The relative Frobenius yields the isogeny $\Fr_{q}\colon L \to \Frq{L} \cong \Frq{P_{-}}/\Frq{U_{-}}$.
In this way we obtain an algebraic zip datum: 
\begin{definition} \label{zipchi}
 The tuple ${\zd}_{G, \chi} =(G, \Frq{P_{-}}, P_{+}, \Fr_{q})$ is called the \textit{algebraic zip datum associated to $G$ and $\chi$}.
\end{definition}
Note that due to the choice of an isogeny $\varphi=\Fr_{q}$ the condition (FC) is automatically fulfilled in this case.

The associated zip group to this zip datum is denoted by $\ezchi$, and the corresponding quotient stack by $\qstack{\ezchi}{G}$.

\subsection*{Quotient stack $\qstack{\ezchi}{G}$}

Denote by $\mathrm{Transp}_{\ezchi}$ the $k$-scheme \\ $\fiber{(\ezchi \times  G)}{G}{\mu ~ G ~ \mathrm{id}}$, where $\mu$ is given by the $\ezchi$ - group action \ref{zipaction}.

We may think  $\qstack{\ezchi}{G}$ as the stack associated to the $k$-groupoid $\{G/\mathrm{Transp}_{\ezchi}\}$(see \cite[(2.4.3)]{laumon} for details), i.e.:
 For a $k$-scheme $S$ the objects of the $k$-groupoid are the elements of $G(S)$ and  morphisms between two objects $g_{1},g_{2}$ 
are the $S$-valued points of the transporter $\mathrm{Transp}_{\ezchi}(g_{1},g_{2})(S)$ of $\ezchi$-action with the composition given by the multiplication map of $\ezchi$.

The underlying topological space of $\qstack{\ezchi}{G}$ has a following common description \cite{oortstrata}.

 If $k= \bar{k}$,  the underlying set $\overline{\Xi}$ is a finite set of $\ez(\bar{k})$-orbits in $G(\bar{k})$,
and the topology is induced by a partial order $\preceq$ on it:  For two $\ezchi(\bar{k})$-orbits $\orb'$ and $\orb$ one sets  $\orb' \preceq  \orb$ if $\orb' \subset \overline{\orb}$, 
where $\overline{\orb}$ denotes the closure of $\orb$ in $G(\bar{k})$. The open sets in this topology are explicitly defined by the following property: $U$ is open if and only if 
for some $\orb \in \overline{\Xi}$  such that  $\orb' \preceq \orb$ for all $\orb' \in U$ follows that $\orb \in U$.
 
Let now $k$ be an arbitrary field, and  $\Gamma=\aut{\bar{k}/k}$ be the profinite group of $k$-automorphisms of $\bar{k}$. 
Then $\Gamma$ acts on $\ezchi(\bar{k})$-orbits of $G(\bar{k})$ preserving the order.
Therefore, one obtains an induced order on the $\Gamma$-orbits of $\overline{\Xi}$,
 and the underlying topological space of $\qstack{\ezchi}{G}$ is isomorphic to $\Xi:=\overline{\Xi}/\Gamma$ with the quotient topology. 

More specifically, the topological space $\overline{\Xi}$ admits the following geometrical description \cite{paul}.

Let $W:=\mathrm{Norm}_{G}(T)(\bar{k})/T(\bar{k})$  be the Weyl group of $G$, $w_{0}$ be the element of maximal length in $G$, 
 and $R_s$ the corresponding set of simple reflections with respect to $T_{\bar{k}} \subset B_{\bar{k}}$.

Let $K \subset R_s$ be a subset. 
\fussy
We denote by  $W_{K}$ the subgroup of the Weyl group $W$ generated by $K$,
and let (cf. \cite[ch. 2.3]{carter})  
\[{}^{K}W:=\{w \in W: w \text{ \textit{of the minimal length in the right coset} }W_{K}w\}.\] 
\sloppy
Note that the Frobenius isogeny $\varphi\colon G \to G$ induces  an automorphism $\overline{\varphi}$ of the Weyl group $W$.

Let $\theta_{0}$ be the element of minimal length in $W_{J}w_{0} W_{\overline{\varphi}(I)}$.

Let further  $I \subset R_s$ be the type of $P_{+}$ and let $J \subset R_s$ be the type of $\Frqb{P_{-}}$.
Then the restriction  $\overline{\psi}:=\intt{\theta_{0}} \circ \overline{\varphi}\colon W \to W$ induces an 
isomorphism of Coxeter systems $(W_{I}, I)$ and  $(W_{J}, J)$.

For $w', w \in {}^{I}W$ one sets $w' \preceq w$ if there is $u \in W_{I}$ such that $u w'\psi(u)^{-1} \leq w$ 
with respect to the Bruhat order on $W$.  

As shown in \cite[Subsection 3.5]{fzips}  $\overline{\Xi} \isoeq {}^{I}W$ with the topology induced by the partial order $\preceq$.

This quotient stack $\qstack{\ezchi}{G}$ is useful to describe the isomorphism classes of $G$-zips of type $\chi$. 

\subsection*{$G$-zips and $F$-zips}

For an affine $k$-group scheme $\mathcal{G}$  we mean by $\mathcal{G}$-torsor a right $\mathcal{G}_{S}$-torsor over $S$ 
for the fpqc-topology. 
In other words, $\mathcal{G}$-torsor $I$ is a scheme $I$ over $S$ endowed with a right action of  $I \times_{S}  (\mathcal{G}_S) \to I$  
written $ (i,g) \mapsto ig$ such that the morphism $I \times_S \mathcal{G_S} \to I \times_S I$ given by $(i,g) \mapsto (i, ig)$ is an isomorphism, 
and there is a scheme $S'$ and an fpqc-morphism $S' \to S$ such that $I(S') \neq \varnothing$. Remark that the last condition can be omitted 
if the structure morphism $I \to S$ is fpqc. 

Let $\mathcal{H}$ be a closed subgroup scheme of $\mathcal{G}$ over $k$ .  We say that a  $\mathcal{H}$-torsor $J$ is a subtorsor of  $\mathcal{G}$-torsor $J$
if there is an $\mathcal{H}$-equivariant inclusion $J \hookrightarrow I$,  where $\mathcal{H}$ acts on $I$ via restriction of the $\mathcal{G}$-action.

We recall the following definition introduced in \cite{fzips}.
\begin{definition} \label{gzps}
1) A $G$-zip of type $\chi$ over $S$ is a tuple $\underline{I}=\gzipdata{I}{\imath}$, where $I$ is a $G$-torsor over $S$,
 $I_{+} \subset I$  a $P_{+}$ subtorsor, $I_{-} \subset I$  a $\Frq{P_{-}}$ subtorsor, and 
 $\imath\colon \Frq{I_{+}}/\Frq{U_{+}} \isoto I_{-}/\Frq{U_{-}}$ an isomorphism of $\Frq{L}$-torsors.

2) A morphism $\gzipdata{I}{\imath} \to \gzipdata{I'}{{\imath}'}$ of $G$-zips of type $\chi$  consists of $G$ resp. $ P_{+}, \Frq{P_{-}}$ equivariant morphisms
$I \to I'$ resp. $I_{\pm} \to I'_{\pm}$ which are compatible with inclusions and the isomorphisms $\imath$ and ${\imath}'$.
 
\end{definition}

The morphisms of $G$-zips of type $\chi$ over $S$ are isomorphisms, hence such  $G$-zips form a groupoid denoted by $\gzipcat(S)$.

As shown in \cite[Prop. 3.2]{fzips}, the groupoids $\gzipcat(S)$ with the obvious pullback definition form a stack  $\gzipcat$ fibered over the category $\kscheme$ .

The stack  $\gzipcat$ is isomorphic to the stack $\qstack{E_{G, \chi}}{G}$ \cite[Prop. 3.11]{fzips}.

The data coming from many interesting geometric objects in nonzero characteristic carries 
the structure of so-called $F$-zips (see \cite[Section 2]{torsten}).

First, we recall the definitions. 
\begin{definition}
Let $\mathcal{M}$ be a locally free $\os$-module of finite rank.

A \textit{descending filtration} $C^{\bullet}$ of $\mathcal{M}$ is a family  $\{C^{i}\mathcal{M}\}_{i \in \zz}$ of $\os$-submodules of $\mathcal{M}$ which are locally direct summands of  $\mathcal{M}$
 such that $C^{i+1}\mathcal{M} \subset C^{i}\mathcal{M}$ for all $i \in \zz$, and $C^{i}\mathcal{M}=\mathcal{M}$ for $i\ll 0$ and  $C^{i}\mathcal{M}=0$ for $i \gg 0$.

Similarly, an \textit{ascending filtration} $D_{\bullet}$ of $\mathcal{M}$ is a family  $\{D_{i}\}_{i \in \zz}$ of $\os$-submodules of $\mathcal{M}$ which are locally direct summands of  $\mathcal{M}$
 such that $D_{i-1}\mathcal{M} \subset D_{i}\mathcal{M}$ for all $i \in \zz$, and $D_{i}\mathcal{M}=\mathcal{M}$ for $i\gg 0$ and  $D_{i}\mathcal{M}=0$ for $i \ll 0$.
\end{definition}
If $S$ has a finite number of connected components, the conditions of the previous definition imply that the subquotiens $\grd{i}\mathcal{M}:=C^{i}\mathcal{M}/C^{i+1}\mathcal{M}$ and $\gra{i}\mathcal{M}:=D_{i}\mathcal{M}/D_{i-1}\mathcal{M}$
are locally free $\os$-modules that vanish outside a finite index range.   
Note that locally free $\os$-modules endowed with a descending resp. an ascending filtration form a category. 
Objects of these categories are pairs $(\mathcal{M}, C^{\bullet})$ resp.  $(\mathcal{M}, D_{\bullet})$ and  maps are the morphisms of    
$\os$-modules that respect these filtrations. 

\begin{definition} \cite[Definition 6.1.]{fzips} \\
1) An $F$-zip over $S$ is a tuple $\fzip{M}=\fzipdata{M}{\varphi}$, where $\mathcal{M}$ is a locally free $\os$-module of finite rank, $C^{\bullet}$
 a descending filtration of $\mathcal{M}$, $D_{\bullet}$ an ascending filtration of $\mathcal{M}$, and $\varphi_{i}\colon \Frqb{\grd{i}\mathcal{M}} \isoto \gra{i}\mathcal{M}$
are $\os$-linear isomorphisms.

2) A \textit{homomorphism } $f\colon \fzip{M}=\fzipdata{M}{\varphi} \to \fzip{N}=\fzipdata{N}{{\varphi}'}$ of $F$-zips over $S$ is a homomorphism 
of the underlying $\os$-modules $\fzip{M} \to \fzip{N}$ satisfying for all $i \in \zz$ the constraints $f(C^{i}\mathcal{M}) \subset C^{i} \mathcal{N}$ and 
$f(D_{i}\mathcal{M}) \subset D_{i} \mathcal{N}$ and making  the following diagram commute:

\begin{displaymath}
\xymatrix{
  \Frqb{\grd{i}\mathcal{M}}  \ar[r]^-{\varphi_{i}}_-{\sim} \ar[d]^{\Frqb{\grd{i}f}} &  \gra{i}\mathcal{M} \ar[d]^{\gra{i}f}\\
  \Frqb{\grd{i}\mathcal{N}}    \ar[r]^-{{\varphi_{i}}'}_-{\sim} &  \gra{i}\mathcal{N}
}
\end{displaymath}
\end{definition}

The resulting category of $F$-zips over $S$ is denoted by $\fzipcat(S)$. 
Its simplest objects are so-called Tate $F$-zips.  
\begin{example}
 The Tate $F$-zip  of weight $d \in \zz$ is the $F$-zip $\tatezip{d}=\fzipdatanew{\os}{\varphi}$, where

\begin{align}
 C^{i}= \left \{ 
\begin{array}{ll}
           \os & \text{ for } i \leq d ,\\
            0  & \text{ for } i>d 
\end{array}
 \right. 
& \text{and }
 D_{i}= \left \{
 \begin{array}{ll}
     0  & \text{ for } i<d,\\
     \os & \text{ for } i \geq d 
\end{array}
 \right. 
 \end{align}
\end{example}
with $\varphi_{d}$ is the identity on $\os=\Frqb{\os}$.
  
With the natural definition of the tensor product and the duals \cite[Section 6]{fzips} the $\fq$-linear category $\fzipcat(S)$ becomes a 
rigid tensor category with the unit object $\tatezip{0}$.

\begin{definition}
 An $F$-zip is called of \textit{rank} $n$, or \textit{height} $n$ if its underlying $\os$-module has constant rank $n$.

Let $\underline{n}\colon \zz \to \nato$ be a function with finite support.
An $F$-zip $\fzip{M}$ is called of type $\underline{n}$ if the graded pieces $\grd{i} \mathcal{M}$, or equivalently  $\gra{i}\mathcal{M}$, 
are locally free $\os$-modules of constant rank $n_{i}:=\underline{n}(i)$ for all $i \in \zz$.  
\end{definition}
Let  $k=\fq$, and $S$ be a $\fq$-scheme.
Denote by $\fzipcatn(S)$ be a subcategory of $\fzipcat(S)$ whose objects are $F$-zips of type $\underline{n}$ and morphisms are isomorphisms.
Since $F$-zips consist of quasi-coherent sheaves and the morphisms thereof, they satisfy the effective descent  with respect to 
any fpqc-morphism of $\fq$-schemes $S' \to S $. Therefore, one obtains the category  $\fzipcatn$ fibered in groupoids which is a stack. 

Let $\fzip{M}=\fzipdata{M}{\varphi}$ be an $F$-zip of type $\underline{n}$ over $S$.  

It is immediate from the definition that $\mathcal{M}$ is Zariski locally  isomorphic to the free $\os$-module  $(k^{n})_S = \os^n$,  and 
the filtered $\os$-modules $(\mathcal{M}, C^{\bullet})$ and $(\mathcal{M}, D_{\bullet})$  to   $(k^{n}, C^{\bullet})_{S}$ resp. $(k^{n}, D_{\bullet})_{S}$. 
Moreover, by a change of basis, we can assume that    $\grd{i}k^{n}= \gra{i}k^{n}=k^{n_{i}}$.

Let $P_{+}:= \autm{(k^{n}, C^{\bullet})}$ and  $P_{-}:= \autm{(k^{n}, D_{\bullet})}$.
As  result, we get two opposite parabolic subgroups of  $\GL{k}$ defined over $k=\fq$ together with the isogeny
of their common Levi factor $L=P_{+} \cap P_{-}$ induced by the Frobenius. As usual denote by $U_{+}$ and $U_{-}$ their unipotent radicals.
  
Now on obtains the corresponding $\GLn$-zip $\underline{I}=\gzipdata{I}{\varphi}$ by taking
\[\begin{array}{rl}
I=&\isom{(k^{n})_S}{\mathcal{M}},\\
I_{+}=&\isom{(k^{n}, C^{\bullet})_S}{(\mathcal{M},C^{\bullet})},\\
I_{-}=&\isom{(k^{n}, D_{\bullet})_S}{(\mathcal{M},D_{\bullet})}
\end{array} \]

Forgetting filtrations gives the $P_{\pm}$ equivariant embeddings $I_{\pm} \hookrightarrow I$.
Moreover, the isomorphism 
$\varphi_{\bullet}\colon \Frqb{\grd{\bullet}\mathcal{M}} \isoto \gra{\bullet}\mathcal{M}$ induces an isomorphism of $L$-torsors:
\begin{align*}
 &\Frqb{I_{+}}/U_{+} \isoeq \isom{(\grd{\bullet} k^{n})_S}{\Frqb{\grd{\bullet}\mathcal{M}}} 
\isoto \isom{(\gra{\bullet} k^{n})_S}{\gra{\bullet}\mathcal{M}}\\
 &\isoeq I_{-}/U_{-}.
\end{align*}

As shown in \cite[Subsection 8.1]{fzips}, the assignment of an $F$-zip to $\GLn$-zip, which is $\fq$-linearly functorial and compatible with the pullback,
gives rise to an isomorphism of stacks   $\fzipcatn$ and $\glnzipcat$. Recall that $\glnzipcat \isoeq \qstack{E_{\GLn, \chi}}{\GLn}$. 

The $F$-zips with additional structures can also be translated to $G$-zips for an appropriate reductive group (see \textit{loc. cit.}).

The practical upshot from the above discussion is that the study of isomorphism classes of $F$-zips or $G$-zips of the fixed type reduces to the study of the stack   
$\qstack{\ezchi}{G}$.

\section{Affineness of the orbits} \label{aff}

Let $k=\bar{k}$ throughout this subsection.
Its aim  is to show that under assumption of the condition (FC) the orbits of the action \ref{zipaction} are affine. As we will see further on, 
it implies purity of  $G$-zips. 

\begin{remark}  \label{czdred}
An orbit of the action \ref{zipaction} on $\hat{G}$ in the non-connected setting (see \cite[Definition 3.6.]{fzips}) is a finite scheme theoretically 
disjoint union of locally closed subsets of $\hat{G}$ which are isomorphic to the orbits of the related connected zip data given by the neutral components $G$ of $\hat{G}$, 
and two parabolic subgroups.
Hence, as already noted in the introduction, we can restrict us to the study of the connected case.  
\end{remark}

We first recall some basic facts. 

\begin{theorem} \cite[Lang-Steinberg thm., section 1.17]{carter} \label{ls}
 Let $\mathcal{G}$ be an affine connected algebraic group over $k$, $F\colon \mathcal{G} \to \mathcal{G}$ an isogeny, such that
$\mathcal{G}^{F}=\{ g \in \mathcal{G}: F(g)=g \}$ is finite.  Then the morphism of  $k$-varieties $\mathcal{L}\colon \mathcal{G} \to \mathcal{G},
g \mapsto g^{-1}F(g)$ is surjective. 
In particular, taking for $F$ a Frobenius map satisfies the condition of the theorem. 
\end{theorem}

\begin{remark} \label{lsrem}
Going through the proof  one can see the theorem holds if the above finiteness condition is replaced by the condition that $\lie F$ is nilpotent.  

We suppose now that the conditions of the previous theorem hold. Then: 
\begin{enumerate}
\item[(i)] By composing with the map $g \mapsto g^{-1}$ we conclude $\mathcal{L'}\colon \mathcal{G} \to \mathcal{G},  g \mapsto g{F(g)}^{-1}$ is also surjective.
\item[(ii)] Let $\mathcal{G}$ act on itself by $F$-conjugation, i.e. $(g,x) \mapsto gxF(g)^{-1}$, $g,x \in G$.  
Then  $\mathcal{G}$ coincides with the orbit of $1$, hence this action is transitive.
\end{enumerate}
\end{remark}
The following two easy corollaries will be further useful.
\begin{corollary} \label{lscor}
Let $F\colon \mathcal{G} \to \mathcal{G}$ be an isogeny, and $x_{0}$ be a point of $\mathcal{G}$. The following statements are equivalent: 
\begin{enumerate}
\item[(i)] $\mathcal{L}$ is surjective 
\item[(ii)]  $\mathcal{G}^{F}$ is finite
\item[(iii)] $\mathcal{L}_{x_0}\colon \mathcal{G} \to \mathcal{G}$, $g \mapsto gx_{0}F(g)^{-1}$ is surjective  
\item[(iv)] $\mathcal{G}^{x_{0}, F}:=\{ g \in \mathcal{G}: gx_{0}F(g)^{-1}=x_{0}\}$ is finite
\item[(v)] $\mathcal{L}_{x_{0}}$ is a finite morphism
\end{enumerate}
\begin{proof}
(ii) $\Rightarrow$ (i) is exactly the statement of Lang-Steinberg theorem.  

(i) $\Rightarrow$ (ii): By the previous remark the $F$-conjugation is transitive. 
Then, for the dimension reason, the stabilizers of all points of $\mathcal{G}$, in particular  $1 \in \mathcal{G}$ which is $\mathcal{G}^{F}$, are finite.

(i) $\Leftrightarrow$ (iii): Both statements are equivalent to the transitivity of  $F$-conjugation.

(iii) $\Rightarrow$ (iv): The $F$-conjugation is transitive by the previous remark. Then, for the dimensions reasons, 
the stabilizers  of all points of $\mathcal{G}$, in particular $\mathcal{G}^{x_{0}, F}$, are finite.

(iv) $\Rightarrow$ (iii):  Consider $F'(g):= x_{0}\ F(g) x_{0}^{-1}$. Then by  Lang-Steinberg theorem the map $\mathcal{L}'': g \mapsto g F'(g)^{-1}$ is surjective.
Hence  $\mathcal{L}_{x_0}= \mathcal{L}''x_{0}$ is so as well. 
    
(v) $\Rightarrow$ (iv):   $\mathcal{G}^{x_{0}, F}$ is finite as being a fiber of a finite (in particular a quasi-finite) morphism.

(iv)  $\Rightarrow$ (v): Note that $\mathcal{L}_{x_{0}}\colon \mathcal{G} \to \mathcal{G}$ is a (right) $\mathcal{G}^{x_{0}, F}$-torsor.

By passing to  $\red{G}$ we can assume $G$ is smooth. 
Thus we have the following Cartesian diagram:
\begin{displaymath}
\xymatrix{
  \mathcal{G} \times  \mathcal{G}^{x_{0}, F} \ar[d]^{\pi_1} \ar[r]^>>>>>{\pi_1} & \mathcal{G} \ar[d]^{\mathcal{L}_{x_{0}}}   \\
  \mathcal{G} \ar[r]^{\mathcal{L}_{x_{0}}} & \mathcal{G}
}
\end{displaymath}
Here $\pi_1$ is a projection onto the first factor. The map  $\mathcal{L}_{x_{0}}$ is quasi-finite and surjective by the foregoing part of the proof, hence it is faithfully flat 
as a  morphism of smooth varieties.  The projection $\pi_1$ is clearly finite, hence  by faithfully flat descent $\mathcal{L}_{x_{0}}$ is so as well.

\end{proof}
\end{corollary}

\begin{corollary} \label{lsfin}
Let $F\colon \mathcal{G} \to \mathcal{G}$ be an isogeny. The  $\mathcal{G}$-action on $\mathcal{G}$ by $F$-conjugation (cf.\ref{lsrem}(ii)) is transitive if and only if
 the stabilizers of all points of $\mathcal{G}$ are finite.
Otherwise, there exist infinitely many orbits of  $\mathcal{G}$-action by $F$-conjugation in $\mathcal{G}$. 
\end{corollary}
\begin{proof}
 The first statement is immediate by Corollary \ref{lscor}. Suppose now there are finitely many orbits of the $\mathcal{G}$-action by $F$-conjugation.
Since $\mathcal{G}$ is connected,  one of them lies dense in $\mathcal{G}$. Therefore, for the dimension reason it has a finite stabilizer.
Now by \ref{lscor} follows that the $\mathcal{G}$-action  by $F$-conjugation is transitive.
\end{proof}

Now let   $I \in R_s$ resp. $J \in R_s$  be the types of the parabolic subgroups $\conj{g_0^{-1}}P'$ and  $P$ both containing Borel subgroup $B$.
Recall that $W_{I}$ and $W_{J}$ are subgroup of Weyl group $W$ generated by sets of simple reflections $I$ resp. $J$. 

As is well known, we have the following Bruhat decomposition of $G$: \[G= \verein{w \in \iweilj} \conj{g_0^{-1}}P'wP,\] where  $\iweilj$ is a system
of representatives for $W_{I}\backslash W \slash W_{J}$ in the normalizer of $T$.  By the left translation with $g_0$ it yields the decomposition: 
$G= \verein{w \in \iweilj} P'g_0wP$.

We now fix some arbitrary $w$ as above. The set $P'g_0wP$ being $P' \times P$ orbit in $G$ is locally closed, and obviously stable 
under the $\ez$-action \ref{zipaction}, hence each orbit is contained in exactly one of such pairwise disjoint pieces.
Moreover, it is clear by the definition of $\ez$-action that each orbit in $P'g_0wP$ contains an element of the form $g=g_{0}wl$ for some $l \in L$.

Now consider the right homogeneous space $G/U$ with the action of $P'$ on it given by 
\[ (p',[g]) \mapsto [p'g \varphi(\pr{L'}{p'})^{-1}]. \]

Note that the restriction of the projection $P'\times P \to P'$ gives rise to the surjective morphism  $\ez \to P'$.
Via this morphism  we obtain the action of $\ez$ on  $G/U$ making the quotient map $G \to G/U$ $\ez$-equivariant. 
Thus, we get the faithfully flat morphism bijectively mapping the $\ez$-orbits of $G$ onto the $P'$-orbits of $G/U$. 
Moreover, this map is affine by \cite[ch. 12, Prop.12.3.(3)]{wedh} since $G$ is affine, and, as is well known, the homogeneous space $G/U$ 
is a quasi-projective variety \footnote{$G/U$ is even quasi-affine 
since $U$ is observable in $G$ (cf. \cite[Ch. 10, Observation 2.4.,Thm. 5.4]{santos}).} (see \cite[Ch. 7, Thm. 4.2]{santos}),
in particular separated.

Note that the above morphism  $\ez \to P'$ induces an isomorphism $\iota\colon \stab{\ez}{g_{0}wl} \to  \stab{P'}{[g_{0}wl]}$
with respect to the actions of $\ez$ resp. $P'$.  The inverse map $\iota^{-1}\colon \stab{P'}{[g_{0}wl]} \to \stab{\ez}{g_{0}wl}$ is given by $p' \mapsto (p',\conj{(g_0wl)^{-1}}{p'})$.

Clearly, $\stab{P'}{[g_{0}wl]} \subset P' \cap \conj{g_{0}w}{P}$, so we have  $\stab{P'}{[g_{0}wl]}=\stab{P' \cap \conj{g_{0}w}{P}}{[g_{0}wl]}$.

\begin{remark} \label{redpr}

There is a possibility to pass from an original zip datum $\mathcal{Z}=(G,P,P',\varphi)$ to some new zip datum  $\mathcal{Z}_{n}$
containing the algebraic groups of lower dimensions.  

This reduction process is introduced in \cite[Section 4]{paul}, and  it yields a one-to-one closure preserving
correspondence between the orbits inside a $\ez$-stable piece $P'g_{0}wP$ and the orbits in $L$ with respect 
to the $\ezn$-action.  
The new zip datum is given by $ \mathcal{Z}_{n}:= (L, Q, Q', \psi)$,  where $Q:= \varphi(L' \cap \conj{g_{0}w}{P})$, $Q':=L \cap \conj{w^{-1}g_{0}^{-1}}{P'}$
 are two parabolic subgroups of $L$ together 
with the isogeny \[\psi:=\varphi\circ \intt{g_{0}w}|_{L \cap \conj{w^{-1}g_{0}^{-1}}{L'}}\colon L \cap \conj{w^{-1}g_{0}^{-1}}{L'} \to \varphi(L' \cap \conj{g_{0}w}{L})\]
between their Levi factors. Note that $\lie \varphi=0$ implies $\lie \psi=0$. 

As the dimensions of the algebraic groups get smaller, the reduction process terminates in a finite number of steps.  
 For a terminating zip datum holds  $G=L$. Thus, it must be of the form  $(G,G,G,\varphi), $ and $\ez \cong G$ acts on $G$ by:
$(g,x) \mapsto g x \varphi(g)^{-1}$, $g,x \in G$. 

Since we assume the condition (FC) \ref{fc} it follows that there is
 the only finite number of orbits with respect of $\varphi$-conjugate $G$-action. Then  by Corollary \ref{lsfin} we conclude  that 
$G$ acts transitively with finite stabilizers. 
\end{remark}

This reduction process makes it possible to  give an inductive description of the stabilizers of the point $g_{0}wl \in P'g_{0}wP$.
\begin{lemma} 
There is an exact sequence of algebraic groups:
\begin{equation} \label{stab}
\xymatrix  @C=1.5pc{
 1 \ar[r] &  \ker e_{n} \ar[r] & \stab{P' \cap \conj{g_{0}w}{P}}{[g_{0}wl]} \ar[rr]^>>>>>{e_n} &&  \stab{\ezn}{l} \ar[r] & 1 
 }
\end{equation}
where $e_{n}$ denotes the restriction of the morphism $e_{n}\colon P' \cap \conj{g_{0}w}{P} \to \ezn$ given by $e_{n}(p'):=(\pr{L}{\conj{w^{-1}g_{0}^{-1}}{p'}},\varphi(\pr{L'}{p'}))$.

The reduced group scheme $\red{\ker e_{n}}$ is isomorphic to  $U' \cap \conj{g_{0}w}{U}$.

\end{lemma}

\begin{proof}
 It's immediate from the definition of  $e_{n}$ that
 $U' \cap \conj{g_{0}w}{U} \subset \red{\ker e_{n}} \subset (\ker \varphi \cdot U')  \cap \conj{g_{0}w}{U}$.

Since $\varphi$ is an isogeny between connected algebraic groups, it follows that $\red{\ker \varphi}$ lies in the center of $L'$ (see \cite[5.3.5]{springer}). Thus,
$\red{\ker \varphi}$ lies in some torus of $L'$, hence $\red{\ker \varphi} \cap \conj{g_{0}w}{U}=1$.  As $U' \cap \conj{g_{0}w}{U}$ is smooth,
it implies $\red{\ker e_{n}}=U' \cap \conj{g_{0}w}{U}$.

And finally, the map  $e_{n}$ is faithfully flat (we will see below in the proof of Theorem \ref{purity} that it is split).
\end{proof}

\begin{remark} \label{exactness}
 Let  $\mathcal{G}$ be an affine connected algebraic group over algebraically closed field $k$  and $\mathcal{H} \subset \mathcal{G}$ be a smooth closed connected subgroup.
Furthermore let $\gmodcat{H}$ and  $\gmodcat{G}$ be the categories of rational $\mathcal{H}$- resp.  $\mathcal{G}$-modules.

A closed subgroup  $\mathcal{H}$ is by definition \textit{exact} in $\mathcal{G}$ if the induction functor  $\mathrm{Ind}_{\mathcal{H}}^{\mathcal{G}}\colon \gmodcat{H} \to  \gmodcat{G}$   
is exact.
It is a well known fact that  $\mathcal{G}/\mathcal{H}$ is affine if and only if $\mathcal{H}$ is exact in  $\mathcal{G}$ \cite[ch. 11, Theorems 4.5.and 6.7.]{santos}.
As the induction functor  $\mathrm{Ind}_{\mathcal{H}}^{\mathcal{G}}\colon \gmodcat{H} \to  \gmodcat{G}$ 
is an adjoint functor to the restriction functor   $\mathrm{Res}_{\mathcal{G}}^{\mathcal{H}}\colon \gmodcat{G} \to  \gmodcat{H}$,
and since restriction is transitive, induction is so as well.

Specifying the exact group theoretical conditions such that $\mathcal{H}$ is exact in $\mathcal{G}$ turns out to be a hard problem \cite{cline}.
Nevertheless, there are some easy situations, e.g. $\mathcal{H}$ is a closed subgroup of the unipotent radical $\urad \mathcal{G}$ of $\mathcal{G}$, the case we will study next.

Assume now: $\mathcal{H} \subset \urad \mathcal{G}$. Thus, proving that $\mathcal{H}$ is exact in $\mathcal{G}$ amounts to showing that  $\mathcal{H}$ is exact in $\urad \mathcal{G}$,
and  $\urad \mathcal{G}$ is exact in $\mathcal{G}$. The latter is obvious as the quotient is an affine (reductive) algebraic group.

The exactness of  $\mathcal{H}$ in $\urad \mathcal{G}$ is also clear \footnote{see also \cite[Corollary 2.2]{cline} for another proof}: Since unipotent groups
 have only trivial characters, there exist   
a rational module $M$ and a point $x \in M$ such that $\urad \mathcal{G}/\mathcal{H}$ is an $\urad \mathcal{G}$-orbit of $x$ (cf.\cite[Ch.7, Corollary 3.6.]{santos}).
Therefore $\urad \mathcal{G}/\mathcal{H}$ being an orbit of a unipotent  group in the affine variety $M$ is closed in $M$, and hence $\mathcal{G}/\mathcal{H}$ is affine.
\end{remark}

\begin{prop} \label{hsaffin}
Let $k$ be an algebraically closed field.

 (i) Let $\mathcal{G}$ be an affine algebraic group over $k$, $\mathcal{H} \subset \urad \mathcal{G}$ a closed subgroup. Then the quotient $\mathcal{G}/\mathcal{H}$ is affine.
 
 (ii) Let $\mathcal{P}$ be a parabolic subgroup of some reductive algebraic group over $k$ and  $\mathcal{V}$ 
 be a smooth, connected closed unipotent subgroup of $\mathcal{P}$.
 Then $\mathcal{P}/\mathcal{V}$ is affine if and only if $\mathcal{V} \subset \urad \mathcal{P}$.
\end{prop}
\begin{proof}
ad i) and "$\Rightarrow$" of (ii) are immediate  due to the previous remark.

ad ii) "$\Leftarrow$:"  Let $\mathcal{U}:=\urad \mathcal{P}$ and $\mathcal{V} \not \subset \mathcal{U}$.
 Fix a Levi factor $\mathcal{L}$ of $\mathcal{P}$, and identify $\mathcal{P}=\mathcal{U} \rtimes \mathcal{L}$.
 Denote by $\pi_{\mathcal{L}}: \mathcal{P} \to \mathcal{L}$ the projection onto  $\mathcal{L}$  and by  $\pi_{\mathcal{U}}: \mathcal{P} \to \mathcal{U}$ 
 the projection onto  $\mathcal{U}$.
 Note that $\pi_{\mathcal{L}}$ is a morphism of algebraic groups while $\pi_{\mathcal{U}}$ is only a morphism of underlying varieties.
 
Let $\mathcal{W}= \pi_{\mathcal{L}}(\mathcal{V}) \subset \mathcal{L} \subset \mathcal{P}$, a smooth unipotent group, 
which is not $0$ due to the assumption. 
Consider now the map $\mathcal{P}/\mathcal{W} \to \mathcal{U}$ \markme{induced by $\pi_{\mathcal{U}}$}{hier wird $\pi_{\mathcal{U}}$ benutzt, hier ist etwas heikle Stelle, obwohl es trivial aussieht.}. 
Since the fibers of this map which are isomorphic to $\mathcal{L}/\mathcal{W}$ and hence \textit{not} affine (e.g. cf. \markme{\cite[ch.8 ,Theorem 7.1.]{santos}}{$G/H$ affin, $G$ reduktiv => $H$ reduktiv, 
reduktiv=geom. reduktiv (dies wird implizit benutzt), nicht zwingend zusammenhängend}),  
 $\mathcal{P}/\mathcal{W}$ is \textit{not} affine as well.  Moreover  $\pi_{\mathcal{L}}|_{\mathcal{V}}: \mathcal{V} \to \mathcal{W}$ 
 is a faithfully flat map of (unipotent) algebraic groups, 
and thus \markme{$\mathcal{V} \to \mathcal{W}$ is a torsor with respect to the kernel of this map, which is Zariski trivial over $\mathcal{W}$}{Wir wissen, dass solche Torsors über affine Schema trivial sind: Sie werden im wesentlichen über $H^1(W,O_W)$ klassifiziert für $W=G_a$). Dann $W \cong V \times \ker$ als Schemata, aber $\ker$ is normal in $W$, also $W \cong V \ltimes \ker$}. This implies 
 that there is a section $\mathcal{W} \hookrightarrow \mathcal{V}$,  and we can identify $\mathcal{W}$ with a closed subgroup of $\mathcal{V}$. 
 Hence $\mathcal{W}$ is exact in $\mathcal{V}$, and it
 follows that $\mathcal{V}$ is \textit{not} exact in $\mathcal{P}$, otherwise by transitivity of  induction we might have concluded $\mathcal{P}/\mathcal{W}$ is affine which is a contradiction.
 Hence $\mathcal{P}/\mathcal{V}$ is \textit{not} affine.
\end{proof}

\begin{remark}
From the construction one sees readily that an $\ez$-orbit in $G$ is affine if and only if the corresponding $P'$-orbit in $G/U$ is so. 
In its turn by  \ref{hsaffin}(ii) latter is solely determined by whether the reduced neutral component of stabilizer of its arbitrary point,
which is a unipotent algebraic group in our setting, lies in the unipotent radical $U'$ of $P'$.   
 \end{remark}

\fussy
\begin{theorem} \label{purity}
Suppose the condition (FC) \ref{fc} is verified.
Then the $\ez$- orbits in $G$ are affine. 
\end{theorem}

\begin{proof}
Let $\Orb$ be an  $\ez$-orbit of some element $g_{0}wl \in G$.  As the quotient morphism $G \to G/U$ is affine, 
it suffices to show that the corresponding $P'$-orbit of $[g_{0}wl] \in  G/U$ is affine. 

Consider the inclusion $\neured{\stab{P' \cap \conj{g_{0}w}{P}}{[g_{0}wl]}} \hookrightarrow$ $\stab{P' \cap \conj{g_{0}w}{P}}{[g_{0}wl]}$ 
whose  cokernel is a finite algebraic group over $k$, say $\gamma$. Then  $\Orb$ is isomorphic to the quotient of 
 $P'/\neured{\stab{P' \cap \conj{g_{0}w}{P}}{[g_{0}wl]}}$ by  $\gamma$.  

Thus, we get a finite surjective map  $P'/\neured{\stab{P' \cap \conj{g_{0}w}{P}}{[g_{0}wl]}} \to \Orb$.  
Hence, by Chevalley's theorem \cite[Theorem 12.39.]{wedh} it suffices to investigate the homogeneous 
space $P'/\neured{\stab{P' \cap \conj{g_{0}w}{P}}{[g_{0}wl]}}$ for affineness.

\textit{Claim:} $\neured{\stab{P' \cap \conj{g_{0}w}{P}}{[g_{0}wl]}}$ is a closed subgroup of $U'$.

\textit{Proof of the claim}: We proceed inductively. The  reduction process mentioned in Remark \ref{redpr} terminates if $P=P'=G$. 
In this case the stabilizer is finite by \textit{loc. cit.}.

Therefore the claim obviously holds in this case.

Due to the inductive assumption we have  $\neured{\stab{Q'}{[l]}} \subset \urad Q'=L \cap \conj{w^{-1}g_{0}^{-1}}{U'}$.
Note that exact sequence in \ref{stab} splits, the splitting morphism is given by the restriction of the map $f_{n}\colon \ezn \to  P' \cap \conj{g_{0}w}{P}$ given by
$(q',q) \mapsto \conj{g_{0}w}{q'}$. 

Thus we have $f_{n}(\urad Q') \subset U'$, and $\neured{\stab{P' \cap \conj{n}{P}}{[g_{0}wl]}}$
is the product of two closed subgroups of $U'$.

Therefore, $\neured{\stab{P' \cap \conj{g_{0}w}{P}}{[g_{0}wl]}}$ is a closed connected subgroup of $\urad P'$,
 and it is exact in $P'$ due to the remark \ref{hsaffin}(i), 
hence their quotient is affine.  
\end{proof}                  
\sloppy
If we abandon the condition (FC) \ref{fc} the claim of the previous theorem fails already in simplest cases:
\begin{counterexample}
Consider the zip datum  $\mathcal{Z}=(\GL{2},\GL{2},\GL{2},\mathrm{Id})$. Hence $\ez$-action amounts to the conjugation in  $\GL{2}$.
Observe that the morphism of $k$-varieties $\lambda\colon \GL{2} \to  \aff{2}$  given on $k$-valued points by $\lambda\colon \begin{pmatrix} a & b \\ c & d \end{pmatrix}
\mapsto (a+d, ad-bc)$  is constant on the $\GL{2}$-orbits.

Denote by $\Orb_{1}$ the orbit of the element  $\begin{pmatrix} 1 & 1 \\ 0 & 1 \end{pmatrix} \in \GL{2}$ of dimension $2$, 
 and  by $\Orb_{2}$ the orbit of $\mathrm{Id} \in \GL{2}$, which is just a point.
Note that $\Orb_{1} \cup \Orb_{2}= \lambda^{-1}(2,1)$ is closed in $\GL{2}$, and hence affine.

But $\Orb_{1}$ is not closed in $\GL{2}$, by conjugating by the matrices   $\begin{pmatrix} 1 & 0 \\ 0 & t \end{pmatrix} \in \GL{2}$  ($t \neq 0$)
 we conclude that   $\begin{pmatrix} 1 & t \\0 & 1 \end{pmatrix} \in \Orb_{1}$ for all $t \neq 0$, therefore 
we have $\mathrm{Id} \in \overline{\Orb_{1}} \setminus \Orb_{1}$.

It follows that $\Orb_{1}$ has codimension $2$ in its closure $\overline{\Orb_{1}}=\lambda^{-1}(2,1)$. Thus, $\Orb_{1}$ is not affine, 
otherwise it clearly contradicts  the algebraic version of Hartogs' theorem  \cite[Theorem 6.45.]{wedh}, see also Lemma \ref{pureimm}. 
\end{counterexample}

\section{Purity of $G$-zip stratification} \label{pure_strat}

Let $\underline{I}$ be a $G$-zip of type $\chi$ over $S$.
Recall that the stack $\gzipcat$ is isomorphic to  $\qstack{\ezchi}{G}$ \cite[Prop. 3.11.]{fzips} with underlying set $\Gamma \backslash \weili$.

$\underline{I}$ defines for all $\Gamma$ - orbits $[w]$  locally closed subschemes $\siw \overset{\jmath}{\hookrightarrow} S$
which are loci, where $\underline{I}$ has locally the constant isomorphism class $[w]$. We will recall an exact definition of $\siw$ below in this section.

There is the following set theoretically disjoint decomposition \begin{equation} \label{decom} S= \verein{[w] \in \weili/\Gamma} \siw. \end{equation}

First we explain  how the previous section implies that the immersion $\jmath$ is affine.

We recall construction of  $\siw$ \cite[Subsection 3.6]{fzips}: 

 A $G$-zip $\underline{I}$  over $S$  defines  by Yoneda lemma  the 1-morphism $\zeta\colon h_{S} \to \gzipcat$, 
where $h_{S}$ is the stack associated to the $k$-scheme $S$, i.e. to its functor of points. 

%

Due to \cite[Prop. 2.2.]{fzips} we can consider $[w]$ as a smooth, locally closed algebraic substack of $ \gzipcat$, 
let $\siw$ be the schematic inverse image $\zeta^{-1}([w])$.

\begin{theorem}
The immersion $\jmath\colon \siw \to S$  is affine.
\end{theorem}

\begin{proof}
By Prop. 2.2. \textit{loc.cit.} the $\Gamma$-orbit of $w$ is a locally closed subset 
of underlying topological space $\overline{\Xi}$ of  $\gzipcat \otimes \bar{k}$ with no two elements
 are comparable with respect to $\preceq$.
Therefore, it is locally closed subset of $\overline{\Xi}$, and it  can
be described as a disjoint union of one-point reduced stacks (cf. \cite[Subsection 2.2]{fzips}).
 
Thus, simply by the base change we obtain the following scheme-theoretically disjoint decomposition: $\siw \otimes \bar{k}= \underset{w' \in \Gamma w}{\bigsqcup} \siwcl{w'}$.

Clearly,  $\jmath$ is affine if and only if each $\jmath \otimes id_{\bar{k}}|_{\siwcl{w'}}\colon \siwcl{w'} \to S$ is affine for $w' \in \Gamma w$.

Thus, without loss of generality we can assume $k$ is algebraically closed. 
In this case $\Gamma=1$ and $[w]$ corresponds to a single orbit $\Orb$ of $G$.

The quotient map $G \to  \gzipcat$ is a representable faithfully flat stack morphism.   
Then the affineness of the morphism   $\Orb \hookrightarrow G$ (see \cite[ch. 12, Prop.12.3.(3)]{wedh}) implies by faithfully flat descent 
that the schematic stack morphism $[w] \to \gzipcat$ is affine.

So, we have the following diagram:

\begin{displaymath}
\xymatrix{
  \siw \ar@^{(->}[r]^{\jmath} \ar[d] & S \ar[d]^{\zeta} \\
  [w]   \ar@^{(->}[r] & \gzipcat
}
\end{displaymath}
Hence the morphism  $\jmath\colon \siw \to S$ is affine simply by the definition.
 
\end{proof}

The affineness of the inclusion $\jmath$ implies the following result:

\begin{corollary}
Suppose that  $S$ is a locally noetherian $k$-scheme, $Z$  a closed subscheme of $S$ of codimension $\geq 2$, which contains no embedded components of $S$
that the restriction of $\underline{I}$ to $S \setminus Z$ is fppf locally constant, 
then $\underline{I}$ is  fppf locally constant.   
\end{corollary}
\begin{proof}
That $Z$ contains no embedded components implies that the scheme theoretic closure of $S \setminus Z$ coincides with $S$. 
The claim of Corollary  is then an immediate consequence from the following lemma.
\end{proof}

\begin{lemma} \label{pureimm}
 Let $X$ be a scheme, $Y$ be an locally-noetherian scheme,  $X \hookrightarrow Y$ be an affine immersion.
Denote by $\overline{X}$ the schematic closure of $X$ in $Y$, and let $Z$ be an irreducible component of $\overline{X} \setminus X \neq \varnothing$.
Then $\codim{Z}{\overline{X}}=1$. 
\end{lemma}

\begin{proof}
Since an affine morphism is quasi-compact, $X \hookrightarrow Y$ factorizes through the inclusion $\overline{X} \hookrightarrow Y$,
 and one has an affine open immersion $X \hookrightarrow \overline{X}$ (cf. \cite[Remark 10.31.]{wedh}).

First assume that $\codim{Z}{\overline{X}}=0$. As $Z$ is closed in $\overline{X}$ of codimension $0$ it must be an irreducible component of $\overline{X}$.
But $Z \cap X= \varnothing$, hence $X$ is not dense in $\overline{X}$. This gives a contradiction to the assumption.
  
Suppose there is a component $Z$ of $\overline{X} \setminus X$ such that  $\codim{Z}{\overline{X}} \geq 2$. Replacing $\overline{X}$
by $\Spec \ofun{\overline{X},Z}$ and $X$ respectively by $X \cap \Spec \ofun{\overline{X},Z}$ we can assume that
 $\overline{X} = \Spec A$ for a local noetherian ring $A$ of dimension at least two and $X= \Spec A \setminus \{z\}$, where $z$ is a closed point of the codimension at least two. 
Again replacing $A$ with the quotient $A/\mathfrak{p}$, where $\mathfrak{p}$ is a minimal prime ideal of $A$, 
we can  furthermore assume that $A$ is integral.

Next, we replace $A$ with the completion  $\hat{A}$ which is possible since the morphism $\Spec \hat{A} \to \Spec A$ 
is faithfully flat (cf.\cite[Prop. B.40]{wedh}), i.e. it preserves the codimensions, and  affine morphisms are stable under base change,  
we can assume $A$ is a complete local noetherian ring.  

 Then $\Spec A$ is excellent due to \cite[Theorem 12.51]{wedh}.  This implies that the normalization $\Spec A' \to \Spec A$ is a finite morphism (cf.\cite[Theorem 12.51]{wedh}).

Replacing $A$ with its normalization we can assume that $A$ is normal.  

Eventually, we get an  affine inclusion $X:=\Spec A \setminus \{z\} \hookrightarrow \Spec A$ of normal noetherian schemes.
By an algebraic analogue of Hartogs' theorem \cite[Theorem 6.45.]{wedh} we conclude that  $A \isoeq \Gamma (X, \ofun{X})$ which is clearly a contradiction to $X$ is affine.
   
\end{proof}

\section{Applications}

\subsection{Purity of level-1-stratification}
Let $S$ be an $\fp$ -scheme and $X$ over $S$ be a Barsotti-Tate group. 
Further let $X[1]$ be the corresponding truncated Barsotti-Tate group of the level $1$, i.e. $p$-torsion of $X$ . 
The strata of level-1-stratification of $S$ corresponds to the loci of $S$ where $X[1]$ has a constant isomorphism class.  

In this subsection, we demonstrate an easy way to show the purity of this stratification by utilizing the fact that its covariant Diedonn\'{e} crystal 
carries  an $F$-zip structure. However, this approach cannot be extended to the study of the higher level stratifications since 
the corresponding  Diedonn\'{e} crystals do not  carry an $F$-zip structure any more.  

In the next subsection we will reprove and generalize this result for stratifications of the higher levels 
using explicit construction of certain moduli spaces of Barsotti-Tate groups. 
Despite some redundancy  we intend to show both approaches in this paper.
The first one, presented in this subsection, may be preferable in the case of level-1-stratification due to its simplicity.    
 
We denote by  $X[1]^{\vee}$ the Cartier dual of $X[1]$.

Let $\diedonne (X)$ be its covariant Diedonn\'{e} crystal  and  $\mathcal{M}(X):=\diedonne (X)(S,S,0)$ 
be its evaluation at the trivial object $(S,S,0)$ of crystalline site.

$\mathcal{M}(X)$ is a local free $\os$- module of the rank equal to  the height $h$ of $X$.

Moreover, $\mathcal{M}(X)$ is endowed with an $F$-zip structure in the following way \cite[Subsection 9.3]{fzips}: 

There is an exact sequence (cf. \cite[Corollaire 3.3.5., Proposition 5.3.6]{bbm}) which is functorial in $X$ and compatible with base change $S' \to S$:
\begin{displaymath}
\xymatrix{
  0 \ar[r] & \omega_{X[1]^{\vee}} \ar[r] & \mathcal{M}(X) \ar[r] & \mathrm{Lie}(X[1]) \ar[r] & 0,
}
\end{displaymath}
 where $\omega_{X[1]^{\vee}}= e^{*} \Omega_{X[1]^{\vee}/S}$ is the $\os$-module of invariant differentials of $X[1]^{\vee}$.   

The relative Frobenius $F_{X/S}\colon X \to \Frp{X}$ and the Verschiebung $V_{X/S}\colon \Frp{X} \to X$ give  rise to $\os$-linear homomorphisms
 $\mathcal{F}:= \mathcal{M}(V): \Frp{\mx} \to \mx$ resp.  $\mathcal{V}:= \mathcal{M}(F): \mx \to \Frp{\mx}$.

Note that the roles of the Frobenius and the Verschiebung are switched in the covariant  Diedonn\'{e} theory.

Moreover, 
$\im \mathcal{V} = \ker \mathcal{F} = \Frp{\omega_{X[1]^{\vee}}}$ and $\im   \mathcal{F} = \ker \mathcal{V}$ are local direct summands 
of $\Frp{\mx}$, respectively $\mx$.

One obtains the corresponding $F$ -zip $\fzipnew{\mx}= \fzipdatanew{\mx}{\varphi}$ with a descending filtration 
$C^{0}= \mx$, $C^{1}= \omega_{X[1]^{\vee}}$ and $C^{2}=0$ and an ascending filtration $D_{-1}=0$, $D_{0}=\ker \mathcal{V}$  and $D_{1}=\mx$
with the isomorphisms $\varphi_{0}\colon \Frpb{C^{0}/C^{1}} = \Frp{\mx}/ \ker \mathcal{F} \isotomap{\mathcal{F}} \im \mathcal{F} = \ker \mathcal{V} \isoeq D_{0}/D_{-1}$.
and  $\varphi_{0}\colon \Frpb{C^{1}/C^{2}} = \im \mathcal{V}  \isotomap{V^{-1}} D_{1}/D_{0}$.

As explained in the previous section, an $F$-zip structure gives a $\GL{h}$-zip structure that also gives a decomposition \ref{decom} of $S$.

Recall that there is an equivalence of categories between truncated Barsotti-Tate groups over a perfect field and a the modulo $p$ reductions of 
the covariant  Diedonn\'{e} modules.

Therefore the decomposition pieces $S^{[w]}$ where $\mathcal{M}(X)$ has fppf-locally a constant isomorphism class are exactly the loci 
where $X[1]$ has a constant one.

Now let $S$ be a locally noetherian $\fp$ -scheme. The purity of the inclusion  $S^{[w]} \hookrightarrow S$ implies in particular that whenever $X[1]$ has a constant isomorphism class over $S \setminus S'$ 
where $S'$ is closed of codimension at least two in $S$, it has  a constant isomorphism class over $S$ overall. 

\subsection{Purity of level-$m$-stratifications}
Let $k$ be a perfect field of characteristic $p>0$,  $\btsm$ be the  moduli stack of  $m$-truncated Barsotti-Tate groups of dimension $d$ and codimension $n-d$ over $k$.
That is, for each $k$-scheme $S$  the groupoid $\btsm(S)$ is the category
 of $m$-truncated Barsotti-Tate groups over $S$ as above with morphisms in  $\btsm(S)$ being isomorphisms of truncated Barsotti-Tate groups. 
Moreover, by \cite[Prop 1.8.]{oortstrata} and by base change,  $\btsm$ is a smooth algebraic stack of finite type over $k$.

Our primary goal here is to sketch briefly the construction of a quotient stack closely related to the stack $\btsm$, to relate latter 
to some algebraic zip datum, and to deduce certain purity results. 
 
For a commutative $\fp$-algebra $R$ and $m \in \nat$  denote by $W(R)$ and $W_m(R)$ the ring of Witt-vectors resp. the ring of Witt-vectors of length $m$ with coefficients in $R$.
Furthermore, let $\sigma_{R}\colon W(R) \to W(R)$ (resp. $\sigma_{R}\colon W_m(R) \to W_m(R)$) be the Frobenius map induced by the Frobenius endomorphism (i.e. $p$-power map) on $R$ and $V: W(R) \to W(R)$
(resp. $\sigma_{R}\colon W_m(R) \to W_{m+1}(R)$) is the Verschiebung; so it holds $\sigma_R \circ V=p$.

Let $\grpk$ be a smooth affine group scheme of finite type over $\Spec(W(k))$.
We denote by $W_m(\grpk)$ the smooth affine algebraic group over $k$ which represents the functor $R \mapsto \grpk(W_{m}(R))$. 

Let  now $X_m$ be an $m$-truncated Barsotti-Tate group over $\Spec R$ of dimension $d$ and height $n$. It can be assigned to the truncated 
display of level $m$ over $R$ (see \cite{lau}), which is given by an element of $\GL{n}(W_m(R))$ for a suitable choice of basis.

E. Lau shows in  \textit{loc.cit.}  Lemma 3.5 that the category $\mathrm{Disp}_m(k)$ of truncated displays of level $m$ over $k$  is equivalent to truncated Diedonn\'{e} modules of
level $m$ over $k$. 
Moreover, by fixing the dimension $d$ one has  the description of the corresponding substack $\mathrm{Disp}_{m,d}$
as a quotient stack(see the construction  in the proof of \textit{loc. cit. }  Prop. 3.15) as follows. 
Let $K_m= W_m(\GL{n})$ and define an algebraic group $K_{\mu,m}$ by 
\begin{align*} 
& K_{\mu,m}(R) = 
\left \{ \left (
\begin{array}{c|c}
A_{d \times d} &   B_{d \times n-d} \\ \hline 
C_{n-d \times d} &  D_{n-d \times n-d}                                                                                               
\end{array} \right ) \in K_m  \right \},
 \text{where } A_{d \times d},  C_{n-d \times d},\\ 
& D_{n-d \times n-d} \text{ are sub-matrices over $W_m(R)$,}\\
& \text{and  $B_{d \times n-d}=(b_{ij})$  with $b_{ij} \in V(W_m(R)) \subset W_{m+1}(R) $ of size} \\ 
& \text{as specified by the lower indices}.
\end{align*}

Let $\imath \colon K_{\mu,m} \to K_m$ be the map induced by reduction modulo $p^m$ followed by the inclusion and denote by $\sigma_{\mu} \colon K_{\mu,m} \to K_m$  
the morphism given by \begin{align*}
                       \begin{pmatrix}  A & B \\
                                        C & D  
                       \end{pmatrix}   \mapsto 
                      \begin{pmatrix}
                        \sigma_R (A) & V^{-1}(B) \\
                        p \sigma_R(C) &  \sigma_R(D) 
                       \end{pmatrix}.
                        \end{align*}
Here the dimensions of sub-matrices  are the same as above,  and $p$, $\sigma_R$, $V^{-1}$ are applied entrywise.

The action of  $K_{\mu,m}$ on $K_m$ is given by  $(x,z) \mapsto \imath(x)  z \sigma_{\mu}(x)^{-1}$, $x \in K_{\mu,m}, z \in K_m$.

The related quotient stack is denoted  by  $\qstack{K_{\mu,m}}{K_m}$, and we have $\mathrm{Disp}_{m,d}=\qstack{K_{\mu,m}}{K_m}$.

Unless stated otherwise, for the rest this subsection we will assume that $k=\bar{k}$.
\begin{example}
Let $m=1$.  Consider the  $K_{\mu,1}$ -action on $K_1= \GL{n}$.

Let $P'=\imath(K_{\mu,1})$  and $P$ be an opposite parabolic of $P'$ with respect to the diagonal torus.
Further let $L= P' \cap P$ be their common Levi factor.

Now consider an zip datum  $\zd=(K_1, P, P', \bar{\sigma})$, 
where $\bar{\sigma}\colon L \to L$ is the reduction modulo $p$ of $\sigma_R$.

Note that $K_{\mu,1}$ acts on $K_1$ precisely with the same orbits as $\zd$ on $K_1$.

Thus, the orbits of $K_{\mu,1}$-action  are affine due to Theorem \ref{purity}.


This example  reproves in particular purity of the level-1-stratification.
\end{example}

In order to show the purity of level-$m$-stratification for $m \geq 2$ consider the smooth morphism $K_{m} \to K_{1}$
of algebraic groups over $k$ which is induced by the reduction modulo $p$. Moreover by reducing modulo $p$ one
obtains a $K_{\mu,m}$-action on $K_{1}$ making $K_{m} \to K_{1}$ $K_{\mu,m}$-equivariant. Thus, it maps $K_{\mu,m}$-orbits in 
$K_{m}$ to $K_{\mu,1}$-orbits in $K_{1}$.

Our aim is to establish the affineness of these orbits. 

 The following general fact \cite[Ch. 11, Theorem 8.4.]{santos} will be further useful:
\begin{lemma} \label{slice}
Let  $\mathcal{G}$ be a smooth affine algebraic group over an algebraically closed field $k$  
and $\mathcal{U} \subset \mathcal{G}$ be a closed  smooth connected unipotent subgroup.

Then the homogeneous space $\mathcal{G}/\mathcal{U}$ is an affine variety  if and only if there exist a morphism of varieties $\Phi\colon \mathcal{G} \to  \mathcal{U}$ such that
$\Phi(xu)=\Phi(x)u$ for all $x \in \mathcal{G}$ and $u \in \mathcal{U}$.
\end{lemma}

The proof of the next proposition is largely influenced by \cite[Subsection 5.1.]{vasiu}.
\begin{prop} \label{afflevel}
 $K_{\mu,m}$-orbits in $K_{m}$ are affine for all $m \in \nat$.
\end{prop}

\begin{proof}
 Let $\Orb$ be an $K_{\mu,m}$-orbit in $K_{m}$ that maps to some $K_{\mu,1}$-orbit  $\orb$ in   $K_{1}$.
Let $\grpi_m$ be the stabilizer of some closed point $x$ of $\Orb$ and  $\grpi$ be the stabilizer of the image of $x$ in $\orb$.
As explained  before (cf. remark \ref{exactness}, proof of Theorem \ref{purity}), the orbit $\Orb$ is affine if and only if 
$\neured{\grpi_m}$ is exact in $K_{\mu,m}$.
We also know that  $\neured{\grpi}$ is exact in  $K_{\mu,1}$.

Consider the exact sequence of smooth affine algebraic groups: 
  \begin{displaymath} 
\xymatrix @C=1.5pc{
 1 \ar[r] & \mathcal{N} \ar[r] &  K_{\mu,m} \ar[rr]^{\mod p}& &  K_{\mu,1} \ar[r] & 1 
 }
\end{displaymath}

 where $\mathcal{N}$ is the kernel of  $K_{\mu,m} \to K_{\mu,1}$, which is smooth connected unipotent algebraic group.

Note that   $K_{\mu,m}$ is $\mathcal{N}$-torsor over  $K_{\mu,1}$,  and it is trivial 
being a torsor of an algebraic unipotent group over affine $k$-scheme.
Hence we have $K_{\mu,m} \isoeq \mathcal{N} \rtimes  K_{\mu,1}$.

Pulling back this exact sequence by the inclusion   $\neured{\grpi} \hookrightarrow K_{\mu,1}$ we get the following  
exact sequence of smooth affine algebraic groups:  

  \begin{displaymath} 
\xymatrix  @C=1.5pc{
 1 \ar[r] & \mathcal{N} \ar[r] &  \mathcal{U} \ar[r] &   \neured{\grpi} \ar[r] & 1 
 }
\end{displaymath}

As before we have  $\mathcal{U} \isoeq  \mathcal{N} \rtimes  \neured{\grpi}$, and $\mathcal{U}$ is a smooth unipotent connected algebraic group.
 
Since $\neured{\grpi}$ is exact in $K_{\mu,1}$, there exists by \ref{slice} a morphism of varieties $\Phi\colon \neured{\grpi} \to  K_{\mu,1}$ 
such that $\Phi(xu)=\Phi(x)u$ for all $x \in  K_{\mu,1}$, $u \in \neured{\grpi}$ . It yields the morphism of varieties $\Phi'\colon \mathcal{N} \rtimes  \neured{\grpi} \to  \mathcal{N} \rtimes  K_{\mu,1}$ 
given by $(\mathrm{Id}, \Phi)$. $\Phi'$ obviously satisfies the condition of Lemma \ref{slice}. This implies affineness of the quotient  $K_{\mu,m}/\mathcal{U}$.
Therefore $\mathcal{U}$ is exact in $K_{\mu,m}$.

Now observe that $\neured{\grpi_m} \subset \mathcal{U}$ and hence unipotent. So   $\neured{\grpi_m}$ is exact in $\mathcal{U}$ (cf. remark \ref{exactness}).
Thus we conclude  $\neured{\grpi_m}$ is exact in $K_{\mu,m}$ by transitivity of induction. Hence $\Orb$ is affine. 
\end{proof}

The main result of \cite{lau}, Theorem 4.5. establishes a connection between the stack of truncated displays
and the stack of truncated Barsotti-Tate groups: It implies there exists a smooth morphism $\varLambda\colon \btsm \to  \mathrm{Disp}_{m,d}$ 
which is an equivalence on geometric points. 

Note that the stacks $\mathrm{Disp}_{m,d}$ and $\btsm$ for $m>1$ contain infinitely many geometric points, hence we cannot simply run through arguments
 of Section \ref{pure_strat} passing from the affineness of the orbits to the purity of the corresponding strata.

Nevertheless, a slight modification of the arguments makes it possible. In this part we follow closely to \cite[Subsections 2.2-2.3]{vasiu}.

Let  $X'$ be an object in  $\btsm(k)$, and  $X:= \varLambda X'$.  Note that $(X',k)$ and $(X,k)$ define  the points of $\btsm$ resp. $\mathrm{Disp}_{m,d}$ 
in sense of \cite[Definition 5.2]{laumon}.  The corresponding fppf $1$-morphisms of stacks over $k$ $X' \colon \Spec k \to \btsm$ and $X \colon \Spec k \to \mathrm{Disp}_{m,d}$ 
admit canonical factorizations $\Spec k \overset{\bar{X'}}{\twoheadrightarrow} \mathring{X'} \overset{\imath'}{\hookrightarrow}  \btsm$ 
and $\Spec k  \overset{\bar{X}}{\twoheadrightarrow} \mathring{X} \overset{\imath}{\hookrightarrow}  \mathrm{Disp}_{m,d}$,
where  $\mathring{X'}$ and $\mathring{X}$ are residue gerbes of points $X'$  and $X$, see \cite[Section 11]{laumon}.
Note that $\bar{X'}$ and  $\bar{X}$ are fppf epimorphisms, and $\imath$, $\imath'$ are monomorphisms.

As $\btsm$ and  $\mathrm{Disp}_{m,d}$ are locally noetherian stacks over $k$,  the points $X',X$ are algebraic by \cite[Th\'{e}or\`{e}me 11.3]{laumon}. Thus  $\mathring{X'}$ and $\mathring{X}$ are fppf
gerbes of finite  type over $\Spec k$.

Further let $\mathfrak{X}_m$ be a $m$-truncated  Barsotti-Tate group  over $S$. 
Thus  $\mathfrak{X}_m$ defines the stack morphism  $\zeta_{\mathfrak{X}_m} \colon S \to \btsm$.

Essentially, our situation is summarized by the following diagram with 2-Cartesian squares: 
\begin{displaymath}
\xymatrix{
 S_{X'}^m \ar[r] \ar@^{(->}[d] & \mathring{X'} \ar@^{(->}[d]  \ar[r]  & \mathring{X}  \ar@^{(->}[d] &\Orb \ar[l] \ar@^{(->}[d]\\
   S \ar[r]^{\zeta_{\mathfrak{X}_m}} &  \btsm \ar[r]^{\varLambda}  & \mathrm{Disp}_{m,d}  & K_m \ar[l]
}
\end{displaymath}
 
Let us explain it in more detail:
\begin{enumerate}
 \item[1)] By $\Orb$ is denoted the  $K_{\mu,m}$-orbit  of an arbitrary lift of  $X$ in $K_m$.
 \item[2)] $\Orb$ is the fiber product in the right square of the above diagram by \cite[Exemple 11.2.2]{laumon}.
 \item[3)] The $\Orb$ is affine by Prop. \ref{afflevel} and smooth over $k$, and $\Orb \hookrightarrow  K_m$ is an affine immersion of noetherian schemes 
 since $\Orb$ is affine and $K_m$ separated (cf. \cite[ch. 12, Prop.12.3.(3)]{wedh}).
 \item[4)] The quotient map $K_m \to \mathrm{Disp}_{m,d}$ is smooth and surjective, in particular a faithfully flat stack morphism.
 \item[5)] By faithfully flat descent  $\mathring{X}  \hookrightarrow \mathrm{Disp}_{m,d}$ is representable  by an affine immersion of finite presentation. 
 \item[6)] $\Orb  \to \mathring{X}$ is a smooth and surjective stack morphism by base change.
 \item[7)] $\mathring{X}$ is smooth over $k$ since $\Orb$ is smooth over $k$ and by 6).
 \item[8)] Similarly by smoothness $\varLambda$, $\mathring{X'}$ is smooth over $k$ since  $\mathring{X'} \hookrightarrow  \mathring{X}$ is so, and by 6).
 \item[9)] Consider the fiber product $\fiber{\mathring{X}}{\btsm}{\mathrm{Disp}_{m,d}}$. Since $\varLambda$ is smooth, and by base change of 5) follows that 
  it is a reduced substack of $\btsm$. Moreover, since  $\varLambda$ induces an equivalence on the geometrical points, the middle square commutes, 
  so there is a 1-morphism $\mathring{X} \to \fiber{\mathring{X}}{\btsm}{\mathrm{Disp}_{m,d}}$ which is an isomorphism of the reduced one-point stacks.  
 \item[10)] The \textit{level $m$ stratum $S_{X'}^m$ of $\mathfrak{X}_m /S$ with respect to $X'$ } is defined by the fiber product in the left square of the above diagram. 
 By the definition of  a residue gerbe, the morphism of $k$-schemes $f \colon T \to S$ factors through $S_{X'}^m$ if and only if $f^{*}\mathfrak{X}_m$ is locally for fppf topology isomorphic
to $\fiber{X'}{T}{k}$ as $m$-truncated Barsotti-Tate groups over $T$.
 \item[11)] $S_{X'}^m \hookrightarrow S$ is an affine immersion of finite presentation by  base change of the morphism in 4). In particular one can view  
$S_{X'}^m$ as a subscheme of $S$.
\end{enumerate}

 Let $k$ be again an arbitrary perfect field of characteristic $p>0$.  In this case we can also define the level $m$ stratum $S_{X'}^m$ as in 9).

Then we have:
\begin{theorem}
The immersion $S_{X'}^m \hookrightarrow S$ is affine.
\end{theorem}
 \begin{proof}
  By base change and by 11), we conclude that  $\basechg{S_{X'}^m}{\bar{k}} \hookrightarrow \basechg{S}{\bar{k}}$ is an affine immersion of finite presentation, and that it by 
faithfully flat descent implies that   $S_{X'}^m \hookrightarrow S$ is so as well and, in particular, pure.
 \end{proof}

%
%

\section{$F$-zip structures on de Rham cohomology}
Let $S$ be an $\fp$-scheme throughout this subsection.

A vast amount of geometric examples of $F$-zips comes from the structures which naturally arise on the de Rham cohomology. 

For an arbitrary $\fp$ scheme $Y$ denote by  $F_Y\colon Y \to Y$  the absolute Frobenius.  

Furthermore let $X$ be a smooth proper scheme over $S$, and  denote by  $f\colon X \to S$ a structure morphism, and by $F=F_{X/S}\colon X \to \Frp{X}$  
the relative Frobenius.

Thus,  we have the following commutative diagram.
\begin{displaymath}
\xymatrix{
X \ar[rd]^{F=F_{X/S}} \ar@/^2pc/[rrd]^{F_X} \ar@/_2pc/[rdd]^{f} & & \\
& \Frp{X} \ar[r]^{\sigma_X} \ar[d]^{\Frp{f}} &  X \ar[d]^{f} \\
& S  \ar[r]^{F_{S}} & S
}
\end{displaymath}


We start with a recollection  the basic facts (cf. \cite[Sect. 6]{ben}, \cite[Subsect. 1.1.]{torsten}).
 
The de Rham cohomology $\hdr{\bullet}{X/S}:= \mathbf{R}f_{*}\kd{\bullet}{X}{S}$ is the hypercohomology of the complex
 $\kd{\bullet}{X}{S}$ with respect to the left exact functor $f_{*}$  going from the category of $\ofun{X}$-modules to the category of $\os$-modules.

Note that the coboundary maps of $\kd{\bullet}{X}{S}$ are $f^{-1}(\os)$-linear but not $\ofun{X}$-linear. 

There are two exact sequences converging to  $\hdr{\bullet}{X/S}$, namely the Hodge-de-Rham sequence
 \[ \hdss{a}{b}= R^{b}f_{*}(\kd{a}{X}{S}) \Rightarrow  \hdr{a+b}{X/S}\]
and the conjugate spectral sequence
 \[ \conjss{a}{b}= R^{a}f_{*}(\mathcal{H}^{b}(\kd{\bullet}{X}{S})) \Rightarrow  \hdr{a+b}{X/S}.\]

Recall, there is  a morphism of the graded $\ofun{\Frp{X}}$-algebras: \[\gamma\colon \osum{i \in \nato}{\kd{i}{\Frp{X}}{S}} \to \osum{i \in \nato}{\mathcal{H}^{i}F_{*} \kd{\bullet}{X}{S}}.
 \] Moreover, if $f$ is smooth, $\gamma$ is an isomorphism denoted by $C^{-1}$ and called the\textit{ Cartier isomorphism} (cf. \cite[Section 3]{luc}).

In addition, it has the following properties:
\begin{enumerate}
 \item[(i)] $C^{-1}$ restricts on the  zero-graded piece to the algebra isomorphism $F^{*}\colon \ofun{\Frp{X}} \to F_{*} \ofun{X}$.
 \item[(ii)] $C^{-1}$ maps $d(\sigma^{-1}(x)) \in \kd{1}{\Frp{X}}{S}$  to the class of $x^{p-1}dx$ in ${\mathcal{H}^{1}F_{*} \kd{\bullet}{X}{S}}$.
\end{enumerate}

Note, that $C^{-1}$ induces an isomorphism of $\os$-modules
\begin{align*}
 &R^{a}\Frp{f_{*}}\kd{b}{\Frp{X}}{S} \isoto R^{a}\Frp{f_{*}}(\mathcal{H}^{b}F_{*} \kd{\bullet}{X}{S}) 
= R^{a}\Frp{f_{*}}F_{*}(\mathcal{H}^{b} \kd{\bullet}{X}{S})=\\
&R^{a}f_{*}(\mathcal{H}^{b} \kd{\bullet}{X}{S})=\conjss{a}{b}
\end{align*}

Moreover, if we assume that $\os$-modules  $R^{a}f_{*} \kd{b}{X}{S}$ are flat (this holds in particular if they are locally free), then we have:

\[R^{a}\Frp{f_{*}}\kd{b}{\Frp{X}}{S} \isoeq R^{a}\Frp{f_{*}}\sigma_{X}^{*} \kd{b}{X}{S} \isoeq  F_S^* f_* R^{a}\kd{b}{X}{S} \isoeq \Frpb{\hdss{b}{a}}\]

Thus under this assumption we get an isomorphism: 
\begin{equation} \label{ssrel}
             \varphi^{ab}\colon \Frpb{\hdss{b}{a}} \isoto  \conjss{a}{b}
\end{equation}
 
\begin{remark} \label{hfcf}
 Fix an integer $n \in \nato$. The definition of spectral sequence, applied to the case
 of Hodge-de Rham spectral sequence implies that the limit term $M:=\hdr{n}{X/S}$ is endowed with a descending filtration $\mathrm{Fil}^{\bullet}$  
such that $\mathrm{Fil}^{k}M/\mathrm{Fil}^{k+1}M \isoeq \hdssinf{k,}{n-k}$. 
We call $\mathrm{Fil}^{\bullet}$ the\textit{ Hodge filtration}.

On the other hand, the conjugate spectral sequence furnishes us with the second descending filtration $\mathrm{Fil'}^{\bullet}$ such that $\mathrm{Fil'}^{k}M/\mathrm{Fil'}^{k+1}M \isoeq \conjssinf{k,}{n-k}$.

Note that $\hdssinf{k,}{n-k}$ and  $\conjssinf{k,}{n-k}$ are in general the subquotients of $\hdss{k,}{n-k}$ resp. $\conjss{k,}{n-k}$. 
\end{remark}

In the classical situation, there are some discrete conditions for degeneration at $E_1$, which arise directly from the construction of the spectral sequence:  

\begin{remark} \label{degeqk}
 Let $K$ be an arbitrary field. Suppose $\bar{X}$ is a proper scheme over $K$, and let $b_n:= \Dim{K} \hdr{n}{\bar{X}}$. We define  Hodge numbers $h^{a,b}$ for $a,b \in \nato$  by  
  and $h^{a,b}:= \Dim{K} H^{b}(\bar{X},\kd{a}{\bar{X}}{K})$, $a,b \in \nato$.
They satisfy the following inequalities: $b_n\leq \nsum{a,b \in \nato, a+b=n} h^{a,b}$ for all $n \in \nato$, 
 and the Hodge-de Rham spectral sequence degenerates in $E_{1}$ if and only if  the latter inequalities are equalities for all $n \in \nato$ . 
 
\end{remark}

The following remark  generalizes the above one: 
\begin{remark} \label{degeq}
Let $R$ be a commutative ring, and $f \colon \tilde{X} \to \tilde{S}$ be a proper smooth scheme over $\tilde{S}:=\Spec R$. Denote by  $\rmodart{R}$  the  category of $R$-modules of finite length. 
We recall  that   $\rmodart{R}$ is an abelian category, and its objects are  both Noetherian and Artinian $R$-modules, or equivalently, 
finitely generated and Artinian ones.  

Length  $\lgth \colon \rmodart{R} \to \nato$  is  additive on exact sequences of objects in  $\rmodart{R}$: Hence for two objects $M$ and $N$ 
in $\rmodart{R}$ such that $N$ is proper subquotient of $M$ holds: $\lgth N < \lgth M$. 

Suppose  that  $\hdss{a}{b}= R^{b}f_{*}(\kd{a}{\tilde{X}}{\tilde{S}})$ and $\hdr{n}{\tilde{X}/\tilde{S}}$  are objects in   $\rmodart{\ofun{\tilde{S}}}$  for all $a,b \in \nato$, 
e.g. it is a case if $\ofun{\tilde{S}}$ is an Artinian ring.

In view of the above and Remark \ref{hfcf} we arrive at the following criterion for  degeneration  the Hodge-de Rham spectral sequence at $E_1$:

We have:
\[\lgth \hdr{n}{\tilde{X}/\tilde{S}} \leq \nsum{a,b \in \nato, a+b=n} \lgth R^{b}f_{*}(\kd{a}{\tilde{X}}{\tilde{S}}),\]
and the Hodge-de Rham spectral sequence degenerates in $E_{1}$ if and only if   the latter inequalities are equalities for all $n \in \nato$.
\end{remark}

Now we recall the following definition.
\begin{definition}
Let $f\colon X \to S$ be a smooth proper morphism of arbitrary schemes. We say $f$ \textit{satisfies  condition} (D) if the following two conditions hold:
\begin{enumerate}
 \item[(a)] The $\ofun{S}$-modules $\hdss{a}{b}= R^{b}f_{*}(\kd{a}{X}{S})$ are locally free of finite rank for all $a,b \in \nato$. 
 \item[(b)] The Hodge-de Rham spectral sequence degenerates at $E_{1}$.
\end{enumerate}
 \end{definition}

\begin{remark} \label{ssbasech}
The part (a) of condition (D) implies that the formation of $\hdss{a}{b}$ commutes with an arbitrary base change $S' \to S$.

The condition (D) remains true after an \textit{arbitrary} base change $S' \to S$ (see \cite[2.2.1.11]{katz}). 
\end{remark}

\begin{remark} \label{Dprop}
Condition (D) implies the isomorphism \ref{ssrel}, and the conjugate spectral sequence also degenerates at $E_{2}$ (see \cite[Proposition 2.3.2]{katz}).  
\end{remark}
  
Since, by a general principle, the formation of $\hdssr{r}{a}{b}$ commutes with any \textit{flat} base change,
 and a condition of the degeneration at $E_1$  expressed as $\hdssr{1}{a}{b}=\hdssr{2}{a}{b}$ 
for all $a,b \in \nato$ is stable under faithfully flat descent, holds the following:  
\begin{remark} \label{fpqcdeg}
Let $f\colon X \to S$ as in the above definition, and $S' \to S$ be an fpqc morphism.  Then  $f_{S'}\colon \fiber{X}{S'}{S} \to S'$ satisfies (D) iff $f$ satisfies (D).  
\end{remark}

These nice properties of the spectral sequence provided $f$ satisfies condition (D) give birth  the following  $F$-zip structure
on $\hdr{\bullet}{X/S}$.

\begin{construction}
 Fix an integer  $0 \leq n \leq 2 \dim(X/S)$.

Suppose  $f\colon X \to S$ satisfies condition (D).   We associate to $f$ an $F$-zip $\fzipdata{M}{\varphi}$
as follows: Set  $M=\hdr{n}{X/S}$.  As the Hodge-de Rham spectral sequence degenerates at $E_{1}$ we have $\hdssinf{k,}{n-k}=\hdss{k,}{n-k}$,
by Remark \ref{Dprop} we also have $\conjssinf{k,}{n-k}=\conjss{k,}{n-k}$.

Let a descending filtration $C^{\bullet}$ be the Hodge filtration, and we define an ascending filtration $D_{\bullet}$ by 
 $D^{i}M= \mathrm{Fil}'^{n-i}M$, $i \in \zz$, where $\mathrm{Fil}'$ is defined as in Remark \ref{hfcf}. 

Note that $\varphi$ is given by the isomorphisms \ref{ssrel} simply by setting $\varphi_{i}= \varphi^{n-i,i}$. 
  
\end{construction}

\begin{definition} \label{liftchar0}
Let $f\colon X \to S$ be a smooth proper morphism.

We say $\tilde{X}$ is a \textit{lift of} $X$ \textit{in zero characteristic} if there exist a  scheme  $\tilde{S}$ 
\textit{flat} over $\Spec \zz_{p}$,  and a scheme morphism $S \to  \tilde{S}$ such that one has the following diagram with a \textit{Cartesian} square:

\begin{displaymath} \label{dg}
\xymatrix{
X \ar[d]^{f} \ar[r] & \tilde{X} \ar[d]^{\tilde{f}}  \\
S   \ar[r] &  \tilde{S} \ar[d]^{flat} \\
   &  \Spec \zz_{p} 
 }
\end{displaymath}

\end{definition}

Presuming the existence of a lift in zero characteristic we are looking for easily testable sufficient conditions under which the condition (D) is met.
First we will need a few technical facts which essentially only rephrase the content of \cite[Ch.2, §5]{mumf}.
\begin{lemma} \label{cohc}
 Let $f\colon X \to S$ be a proper smooth morphism of locally noetherian schemes with $S= \Spec A$ affine. 
Furthermore let $B$ be an arbitrary $A$-algebra, and $Y= \Spec B$. 
Then there is a finite complex $0 \rightarrow \mathcal{F}^0 \rightarrow \mathcal{F}^1 \rightarrow \ldots \rightarrow \mathcal{F}^m  \rightarrow 0$
of finitely generated projective $A$-modules such that one has for all $n \in \nato$ the natural isomorphism of $B$-modules:
$\hdr{n}{\fiber{X}{Y}{S}/Y} \cong H^{n}(\ten{\mathcal{F}^{\bullet}}{B}{A})$.
\end{lemma}
\begin{proof}
Let $\mathfrak{U}=\{U_{i}\}_{i \in I}$ be a finite affine cover of $X$, and consider the finite Čech bicomplex $\check{C}^{\bullet \bullet}= \check{C}^{\bullet}(\mathfrak{U}, \kd{\bullet}{X}{S})$.
Further let $\mathcal{F}^{\bullet}$ be a total complex associated to the bicomplex $\check{C}^{\bullet \bullet}$.  As $\kd{a}{X}{S}$  are locally free $\ofun{X}$-modules, and since $X$ is flat over $S$,
they are flat over $S$. Moreover, as $f\colon X \to S$ is separated, hence $\mathcal{F}^{\bullet}$ is a complex of flat $A$-modules, which represents 
 the complex   $\mathbf{R}f_{*}\kd{\bullet}{X}{S}$ in the derived category. Moreover, by \cite[§5 Lemma 1,2]{mumf} we can 
 assume that $\mathcal{F}^{\bullet}$ is a complex of finitely generated projective $A$-modules, and for all $A$-algebras $B$, $\{\fiber{U_{i}}{Y}{S}\}_{i \in I}$ is the cover of $\fiber{X}{Y}{S}$, and  $\ten{\check{C}^{\bullet}(\mathfrak{U}, \kd{\bullet}{X}{S})}{B}{A}$ is the corresponding Čech bicomplex.
The associated total complex is just $\ten{\mathcal{F}^{\bullet}}{B}{A}$, and so we have:
$\hdr{n}{\fiber{X}{Y}{S}/Y}=\mathbf{R}^{n}f_{*}\kd{\bullet}{\fiber{X}{Y}{S}}{Y}\cong  H^{n}(\ten{\mathcal{F}^{\bullet}}{B}{A})$  as required. Moreover, this isomorphism is obviously functorial in $B$.
\end{proof}
%
%

The previous lemma  leads us to the following semi-continuity result for the dimension of the cohomology groups of the fibers 
whose proof follows verbatim along the same lines as \cite[Ch. II, §5, Corollary]{mumf}.
For the reader's convenience, we will sketch it.
\begin{prop} \label{semicont}
Let $f\colon X \to S$ be a proper smooth morphism of locally noetherian schemes.  Then  we have:

i) For each $n \in \nato$,  the function $S \to \zz$ defined by  $s \mapsto \Dim{\kappa(s)}\hdr{n}{X_s}$ is upper semi-continuous.

ii) The Euler characteristic $\chi\colon S \to \zz$, $s \mapsto \nsum{j \in \nato}(-1)^{j}\Dim{\kappa(s)}\hdr{j}{X_s}$ is locally constant on $S$.

\end{prop}
\begin{proof}
 
Since the question is local on $S$, we may assume $S=\Spec A$, where $A$ is a local ring.  Since all projective modules over a local ring are free, we can pick a complex $\mathcal{F}^{\bullet}$ 
of finitely generated free $A$-modules which furnishes us with the isomorphism in Lemma \ref{cohc}. Let $d^i\colon \mathcal{F}^i \to\mathcal{F}^{i+1}$ be the coboundary maps of $\mathcal{F}^{\bullet}$.  
Then by the previous lemma holds:
\begin{align*}  \label{easycalc}
& \Dim{\kappa(s)}\hdr{n}{X_s}= \Dim{\kappa(s)} \ker \ten{d^{n}}{\kappa(s)}{A}- \Dim {\kappa(s)} \im \ten{d^{n-1}}{\kappa(s)}{A}= \\ \tag{\dag}
& \Dim{\kappa(s)} \ten{\mathcal{F}^{n}}{\kappa(s)}{A}-\Dim {\kappa(s)} \im \ten{d^{n}}{\kappa(s)}{A}- \Dim {\kappa(s)} \im \ten{d^{n-1}}{\kappa(s)}{A}.
\end{align*}
 
$\Dim{\kappa(s)} \ten{\mathcal{F}^{n}}{\kappa(s)}{A}$ is constant in $s$; therefore it amounts to show that the function $\rho\colon S \to \zz$, $s \mapsto \Dim {\kappa(s)} \im \ten{d^{i}}{\kappa(s)}{A}$ is
 lower semi-continuous for each $i \in \nato$, i.e. the set $M_{\rho}=\{s \in S: \rho(s)< r\}$ is closed in $S$ for each $r \in \nato$. 

Consider now the $A$-linear map $\wedge^r d^{i}\colon \bigwedge ^{r}K^{i} \to  \bigwedge ^{r}K^{i+1}$ between free $A$-modules of finite rank induced by $ d^{i}$. Clearly, then we have:  
 $M_{\rho}= \{s \in S: \ten{\wedge^r d^{i}}{\kappa(s)}{A}=0 \}$. Moreover, the  map $\wedge^r d^{i}$ is  given by a matrix with entries in $A$, 
which correspond to the global sections  of the structure sheaf on $S$. Their common  zero locus defines a closed set in $S$. 

The second assertion follows on taking alternating sum of \ref{easycalc} over $j$.    
\end{proof}

The following proposition \textit{loc. cit.} will be further useful.
\begin{prop} \label{fibeq}
  Let $f\colon X \to S$ be a proper morphism of locally noetherian schemes, and  $F$ a coherent sheaf on $X$, flat over $S$. Assume $S$ is reduced.
For each $b \in \nato$ the following conditions are equivalent:
\begin{itemize}
 \item[(i)]  $s \mapsto  \Dim{\kappa(s)} H^{b}(X_s, F_s)$ is a locally constant function on $S$.
 \item[(ii)]  $R^{b}f_*(F)$ is a locally free sheaf on $S$, 
and for all $s \in S$, the natural map $\ten{R^{b}f_*(F)}{\kappa(s)}{\os} \longrightarrow  H^b(X_s,F_s) $ is an isomorphism.
\end{itemize}
\begin{proof}
The claim is a part of \cite[Ch.2, Corollary 2]{mumf}.
\end{proof}
\end{prop}
 
The Mumford's proof of Proposition \ref{fibeq} relies solely on the fact that $H^{b}(\fiber{X}{\Spec B}{\Spec A}/\Spec B, \ten{F}{B}{A})$ with $A$ and $B$ 
as in Proposition \ref{cohc} can be computed as $H^{b}(\ten{\mathcal{F}^{\bullet}}{B}{A})$, where $F^{\bullet}$ is a perfect complex of $A$-modules.
By Proposition \ref{cohc} the same holds true for $\hdr{n}{\fiber{X}{Y}{S}/Y}$. Thus we have:

\begin{corollary} \label{drconst}
Let $f\colon X \to S$ be a smooth proper morphism of locally noetherian schemes, and $S$ be reduced. 
For each $n \in \nato$ the following conditions are equivalent:
\begin{itemize}
 \item[(i)]  $s \mapsto  \Dim{\kappa(s)} \hdr{n}{X_{s}}$ is a locally constant function on $S$.
 \item[(ii)]  $\hdr{n}{X/S}$ is a locally free $\os$-module, 
and for all $s \in S$, the natural map $\ten{\hdr{n}{X/S}}{\kappa(s)}{\os} \longrightarrow \hdr{n}{X_{s}}$ is an isomorphism.
\end{itemize}

\end{corollary}

The constancy of Hodge numbers in fibers of the lift in zero characteristic turns out to be a sufficient for the condition (D) to be met:   
\begin{prop} \label{consthn}
Let $f\colon X \to S$ be a smooth proper morphism. 
Suppose that there is a lift of $X$ in zero characteristic, $\tilde{f}\colon \tilde{X} \to \tilde{S}$  such that 
$\tilde{X}$ and $\tilde{S}$ are locally noetherian schemes, $\tilde{f}$ is proper and smooth, and $\tilde{S}$ reduced. 

 Further assume the functions $\tilde{s} \mapsto \Dim{\kappa(\tilde{s})} H^{b}(\tilde{X}_{\tilde{s}}, \kd{a}{\tilde{X}_{\tilde{s}}}{\kappa(\tilde{s})})$
are locally constant on $\tilde{S}$ for all $a,b \in \nato$.

Then $f$ satisfies condition (D).
\end{prop}
\begin{proof}
It suffices to prove the assertion of the proposition for $\tilde{S}$ is connected. 
We complete the right column of diagram \ref{dg} with Cartesian squares by taking generic fibers of $\tilde{X}$, $\tilde{S}$, $\Spec \zz_{p}$:
\begin{displaymath} \label{dgext}
\xymatrix{
X \ar[d]^{f} \ar[r] & \tilde{X} \ar[d]^{\tilde{f}} & X' \ar@_{(->}[l] \ar[d]^{f'}\\
S   \ar[r] &  \tilde{S} \ar[d]^{flat} &  S' \ar@_{(->}[l]_{\jmath} \ar[d]\\
 &  \Spec \zz_{p} & \Spec \qp \ar@_{(->}[l] 
}
\end{displaymath}
Note that $\basechg{\kd{a}{\tilde{X}}{\tilde{S}}}{\tilde{s}} \cong \kd{a}{\tilde{X}_{\tilde{s}}}{\kappa(\tilde{s})}$ (e.g. \cite[1.3]{luc}) for all $a \in \nato$.

Since the open immersions are preserved by a base change, $\jmath\colon S' \hookrightarrow \tilde{S}$ is an open immersion again; moreover $\jmath$ is dominant  by \cite[IV/2, Proposition 2.3.7(i)]{EGA} because $\Spec  \qp \hookrightarrow \Spec \zz_{p}$
is dominant and quasi-compact.  

By Prop. \ref{fibeq} $R^{b}\tilde{f}_{*}(\kd{a}{\tilde{X}}{\tilde{S}})$ are locally free
 modules of finite rank for all $a,b \in \nato$. In particular, their formation commutes with an arbitrary base change.

Thus  $R^{b}f'_{*}(\kd{a}{X'}{S'}) \isoeq \jmath^{*}R^{b}\tilde{f}_{*}(\kd{a}{\tilde{X}}{\tilde{S}})$ are locally free  $\ofun{S'}$-modules of the same rank. 
Note that the last isomorphism is true in general since $\jmath$ is a flat morphism.


As $S'$ is of zero characteristic  the corresponding Hodge-de Rham sequence degenerates at $E_1$ by \cite{dlg}. 
 On the other hand, by Proposition \ref{semicont} the function $\tilde{s} \mapsto \Dim{\kappa(\tilde{s})}\hdr{n}{X_{\tilde{s}}}$   is upper
semi-continuous on $\tilde{S}$.
  
Let  $s' \in \im \jmath$. As Hodge-de Rham sequence degenerates at $E_1$ in zero characteristic, for $K= \kappa(s')$  holds the equality of Remark $\ref{degeqk}$. 

Since  by Prop. \ref{fibeq} the Hodge numbers are constant on $\tilde{S}$, we have $\Dim{\kappa(\tilde{s})}\hdr{n}{\tilde{X}_{\tilde{s}}} \leq \Dim{\kappa(s')}\hdr{n}{\tilde{X}_{s'}}$.
On the other hand  $\im \jmath$ is open and dense in $\tilde{S}$, so Proposition \ref{semicont} forces the equality for all $n \in \nato$.
Thus  $\tilde{s} \mapsto  \Dim{\kappa(\tilde{s})} \hdr{n}{\tilde{X}_{\tilde{s}}}$ is a constant function on $\tilde{S}$.

Thus Prop. \ref{drconst} implies that  $\hdr{n}{\tilde{X}/\tilde{S}}$ is locally free $\ofun{\tilde{S}}$-module.

In fact, one can test the  degeneration of Hodge-de Rham sequence locally on $\tilde{S}$, moreover it is sufficient to prove it for $\tilde{S}= \Spec R$,
 where $R$ is a local Artinian ring, see the proof of \cite[Proposition 2.3.2]{katz}.

Note that in this case for the only point $\tilde{s} \in \tilde{S}$ holds $\lgth \hdr{n}{\tilde{X}/\tilde{S}}=\Dim{\kappa(\tilde{s})} \hdr{n}{\tilde{X}_{\tilde{s}}} \cdot \lgth R$ and
$\lgth R^{b}\tilde{f}_{*}(\kd{a}{\tilde{X}}{\tilde{S}})=\Dim{\kappa(\tilde{s})} H^b(\tilde{X}_{\tilde{s}},\kd{a}{\tilde{X}_{\tilde{s}}}{\kappa(\tilde{s})} ) \cdot \lgth R$

Thus, by of Remark \ref{degeq}, the  Hodge-de Rham sequence is degenerate for  $\tilde{f}\colon \tilde{X} \to \tilde{S}$.

We see that $\tilde{f}\colon \tilde{X} \to \tilde{S}$ satisfies condition (D) except for $\tilde{S}$ is not in characteristic $p$. But Remark \ref{ssbasech}
remains true also in this case, hence $f\colon X \to S$ also satisfies condition (D). 

 \end{proof}

\subsection{Applications of Proposition \ref{consthn}} 
\subsection*{K3 schemes over $S$}
 
First we recall the definitions (see \cite{rizov}).
\begin{definition}
Let $K$ be an arbitrary field, and $Y$ be an arbitrary base scheme. 

i) A smooth proper geometrically connected scheme $X$ over $K$ of  dimension $2$ is called a \textit{K3 surface} if $\kd{2}{X}{K} \cong \ofun{X}$, and $H^1(X, \ofun{X})=0$.

ii) A \textit{polarization on a K3 surface} $X$ is a global section $\lambda \in \mathrm{Pic}_{X/K}(K)$
that over $\bar{K}$ is the class of an ample line bundle $\mathcal{L}_{\bar{K}}$ . The degree
of  $\mathcal{L}_{\bar{K}}$, which is by definition the selfintersection index $(\mathcal{L}_{\bar{K}},\mathcal{L}_{\bar{K}})_{X_{\bar{K}}}$, 
is called the \textit{polarization degree} of  $\lambda$. A polarization degree is always an even number.

iii) A \textit{K3 scheme} over $Y$ is a scheme $X$ together with proper, smooth morphism  $f\colon X \to Y$ whose geometric fibers are K3 surfaces. 

iv)  A \textit{K3 space} over a scheme $Y$ is an algebraic space $X$ together with a proper and smooth
morphism $f\colon X \to Y$  such that there is an  \'{e}tale cover $Y'\to Y$ of $Y$ for which $f_{Y'}\colon \fiber{X}{Y'}{Y} \to Y'$ is a K3 scheme over $Y'$.

v) A \textit{polarization on a K3 space}  $f\colon X \to Y$ is a
global section $\lambda \in \mathrm{Pic}_{X/Y}(Y)$  such that for every geometric point $\bar{y}$ of $Y$ the section $\lambda_{\bar{y}} \in\mathrm{Pic}_{X_{\bar{y}}/\kappa(\bar{y})}(\kappa(\bar{y}))$
is a polarization of $X_{\bar{y}}$, see \cite[section 1.3.1]{rizov} for a definition of relative Picard functor $\mathrm{Pic}_{X/Y}$.
\end{definition}

We recall also a well-known fact about the Hodge diamond of a K3 surface $X$, see \cite[Proposition 1.1]{k3del}: 
\begin{remark} \label{k3hn}
For the Hodge numbers of $X/K$ holds:\\
$h^{1,0} = h^{0,1} = h^{2,1} = h^{1,2} = 0$;\\
$h^{0,0} = h^{2,0} = h^{0,2} = h^{2,2} = 1$; \\
$h^{1,1} = 20$.

Note that in particular they do not depend on the field $K$.
\end{remark}

J. Rizov constructs in \textit{loc.cit.} a separated Deligne-Mumford stack $\mathcal{M}_{2d}$ fibered over $\catsch$ of polarized pairs $(f\colon X \to Y, \lambda)$, where $X$ 
is a K3 space over $Y$, and $\lambda$ is a  polarization of the (constant) degree $2d$ on it. Moreover he shows in  \cite[Proposition 1.4.15]{rizov} that 
the moduli stack $\mathcal{M}_{2d}$ is smooth of relative dimension $19$ over $\zz[\frac{1}{2d}]$.

This result establishes the following method of the construction of a lift in zero characteristic:
\begin{corollary}
 Let $(f\colon X \to S, \lambda)$ be a polarized K3 scheme with a polarization of degree $2d$. Assume $p \nmid 2d$ .

Then $f$ satisfies condition (D).
\end{corollary}

\begin{proof}
Denote by  $\mathcal{M}^{p}_{2d}= \ten{\mathcal{M}_{2d}}{\zz_{p}}{\zz[\frac{1}{2d}]}$ which is separated, smooth, in particular flat, Deligne-Mumford stack over $\zz_{p}$.
In particular, one has a  \'{e}tale surjection $M_{2d}' \to \mathcal{M}^{p}_{2d}$ with $M_{2d}'$ is a smooth scheme over $\zz_{p}$.

Let $\mathfrak{X} \to \mathcal{M}^{p}_{2d}$ be the universal K3 space with a polarization of degree $2d$. 
 $(f\colon X \to S, \lambda)$  corresponds to the stack morphism $F\colon S \to \mathcal{M}^{p}_{2d}$ such that
one has the following Cartesian diagram: 

\begin{displaymath} \label{univ}
\xymatrix{
X \ar[d]^{f} \ar[r] & \mathfrak{X} \ar[d]  \\
S  \ar[r]^{F} & \mathcal{M}^{p}_{2d}
}
\end{displaymath}

Let $X':=\fiber{X}{M'_{2d}}{\mathcal{M}^{p}_{2d}}$, $S':=\fiber{S}{M'_{2d}}{\mathcal{M}^{p}_{2d}}$, $\mathfrak{X}':=\fiber{\mathfrak{X}}{M'_{2d}}{\mathcal{M}^{p}_{2d}}$
be the \'{e}tale covers of $X$ resp. $S$ resp. $\mathfrak{X}$.

By the base change by $M'_{2d} \to \mathcal{M}^{p}_{2d}$ we obtain the following Cartesian diagram:

\begin{displaymath} \label{univn}
\xymatrix{
X' \ar[r] \ar[d]_{f_{M'_{2d}}} &   \mathfrak{X}' \ar[d]\\
S'  \ar[r]^{F_{M'_{2d}}} &  M'_{2d}
}
\end{displaymath}

Note that  of the diagonal  $\mathcal{M}^{p}_{2d} \to \mathcal{M}^{p}_{2d} \times \mathcal{M}^{p}_{2d}$ of a Deligne-Mumford stack being a schematic morphism 
implies that  $X' \to S'$  is a K3 scheme, and $\mathfrak{X}' \to M'_{2d}$ is a K3 space.

Note that by Remark \ref{fpqcdeg} it suffices to test condition (D) for a K3 scheme $X' \to S'$.

Now by the definition of K3 space we can choose an \'{e}tale cover $M''_{2d} \to M'_{2d}$  such that  $\mathfrak{X}'':=\fiber{\mathfrak{X}'}{M''_{2d}}{M'_{2d}} \to  M''_{2d}$ is a K3 scheme.

By the same reasoning as above  we can test  condition (D) for a K3 scheme  $X'':=\fiber{X'}{M''_{2d}}{M'_{2d}}  \to S'':=\fiber{S'}{M''_{2d}}{M'_{2d}}$. 

But $\mathfrak{X}'' \to M''_{2d}$ is a lift in zero characteristic  of $X''/S''$ in sense of Definition \ref{liftchar0} which clearly  satisfies the assumptions
of Proposition \ref{consthn}. 
This proves the claim. 
\end{proof}

\begin{remark}
By utilizing the existence and the regularity properties of the moduli stack  of polarized abelian schemes over $\zz$, see e.g \cite{djong}, one can show in a similar vein that 
an polarized abelian scheme $X$ over $S$ under certain restrictions also satisfies condition (D).
In fact, the polarization assumption can be abandoned, and condition (D) holds  true for an arbitrary abelian scheme over $S$, see \cite[§ 2.5. Prop 2.5.2]{bbm}. 
 \end{remark}

\subsection*{Lift over $\Spec W(k)$}

Here we list some examples of proper smooth schemes $X$ over $k$  where condition (D) holds, i.e.,  $S=\Spec k$ and $\tilde{S} = \Spec W(k)$:  

\begin{enumerate}
\item[(1)] $X/k$  is a \textit{proper smooth curve}:  In fact, there is a proper smooth lift $\tilde{X} \to \Spec W(k)$, since the obstructions lie in $H^2(X,{\Omega}^{\vee}_{X/k})$,
 and  $H^2(X,\ofun{X})$, see e.g. \cite[III/1, Th\'{e}or\`{e}me 5.1.4]{EGA}. 

By Serre duality for Hodge numbers  holds  $h^{0,0} = h^{1,1}=t$, $h^{1,0} = h^{0,1}=g$, and they are the same on the generic and the special fiber   
since the Euler characteristic $\chi= 2t-2g$ is constant on  $\tilde{S}$ by Proposition \ref{semicont}(ii), and $t$, which is the number of  geometric components,
is also constant by \cite[IV/3, Proposition 15.5.9(ii)]{EGA}.

\item[(2)] $X/k$ is a \textit{K3 surface}:  By \cite{k3del} there exists a lift to a K3 scheme over $W(k)$, 
and the Hodge numbers do not depend on $X$ and on the ground field. 

\item[(3)] $X/k$ is an \textit{Enriques surface} if $\mathrm{char}(k) \neq 2$. In this case $H^{2}(X , \ofun{X})=0$ and it has an \'{e}tale cover by a K3 surface $Y$.
By Serre duality  $H^2(X,{\Omega}^{\vee}_{X/k})= 0$  if and only if $H^0(X,\ten{\Omega_{X/k}}{\omega_X}{})=0$.
The last equality is true since it holds for an \'{e}tale cover $Y$, see \cite[Theorem 1.1.]{lang}. Thus a lift exists for the same reason as in (1).

Note that similar to K3 surfaces the Hodge numbers for Enriques surfaces over $K$ in $\mathrm{char}(K) \neq 2$ do not depend on any choice. 

\item[(4)] $X/k$ is a \textit{smooth complete intersection}  in $\mathbb{P}_k^n$, see \cite[Expos\'{e} XI, Th\'{e}or\`{e}me 1.5]{sga7}.  
%

\item[(5)] $X/k$ is a \textit{smooth proper toric variety}, see \cite{bli}.
\end{enumerate}

\nocite{*}
\bibliographystyle{alpha}
\bibliography{doctor}

\begin{thebibliography}{NVW10}

\bibitem[BBM82]{bbm}
Pierre Berthelot, Lawrence Breen, and William Messing.
\newblock {\em {T}h\'{e}orie de {D}ieudonn\'{e} {C}rystalline {II}}.
\newblock Springer-Verlag, 1982.

\bibitem[Bli01]{bli}
Manuel Blickle.
\newblock {C}artier isomorphism for toric varieties.
\newblock {\em J. Algebra 237}, (16):342--357, 201.

\bibitem[Car85]{carter}
Roger.~W. Carter.
\newblock {\em {F}inite groups of {L}ie type}.
\newblock John Wiley Sons Inc, 1985.

\bibitem[CPS77]{cline}
Edward Cline, Brian Parshall, and Leonard Scott.
\newblock {I}nduced {M}odules and {A}ffine {Q}uotients.
\newblock {\em Math. Ann. 230}, pages 1--14, 1977.

\bibitem[Del]{k3del}
Pierre Deligne.
\newblock {R}el\`{e}vement des surfaces {K}3 en caract\'eristique nulle.
  ({R}edige par {L}uc {I}llusie). {A}ppendice: {C}lasses de {C}hern
  cristallines et intersections (par {P}. {D}eligne et {L}. {I}llussie).
\newblock Surfaces alg\'{e}briques, S\'{e}minaire de G\'{e}om\'{e}trie
  Alg\'{e}brique d'Orsay 1976-78, Lect. Notes Math. 868, 58-75; 75-79 (1981).

\bibitem[DG64]{sga33}
Michel Demazure and Alexandre Grothendieck.
\newblock {\em {S}\'{e}minaire de {G}\'{e}om\'{e}trie {A}lg\'{e}brique du
  {B}ois {M}arie}, volume~3.
\newblock Springer-Verlag, 1962-64.

\bibitem[DI87]{dlg}
Pierre Deligne and Luc Illusie.
\newblock Rel\`evements modulo $p^2$ et d\'{e}composition du complexe de de
  {R}ham.
\newblock {\em Invent. Math.}, 89(2):247--270, 1987.

\bibitem[dJ93]{djong}
Aise~Johan de~Jong.
\newblock {T}he moduli spaces of polarized abelian varieties.
\newblock {\em Math. Ann.}, 295(1):485--503, 1993.

\bibitem[dJO00]{oort}
Aise~Johan de~Jong and Frans Oort.
\newblock {P}urity of the stratification by {N}ewton {P}olygons.
\newblock {\em Math. Soc.}, 13(1):209--241, 2000.

\bibitem[DK69]{sga7}
Pierre Deligne and Nicholas Katz.
\newblock {G}roupes de monodromie en géométrie algébrique -({SGA} 7- vol.
  2).
\newblock In {\em {S}\'{e}minaire de {G}\'{e}om\'{e}trie {A}lg\'{e}brique du
  {B}ois {M}arie}, Lecture notes in mathematics. Springer-Verlag, 1967-69.

\bibitem[Gro]{EGA}
Alexandre Grothendieck.
\newblock {\em {É}lements de G{é}om{é}trie Alg{é}brique}, volume 4, 8, 11,
  17, 20, 24, 28, 32 of {\em Publ. Math. IHES}.
\newblock Springer-Verlag.
\newblock Bures-Sur-Yvette, 1960--1967; see also {\it Grundlehren} {\bf 166}
  (1971).

\bibitem[GW10]{wedh}
Ulrich Görtz and Torsten Wedhorn.
\newblock {\em {A}lgebraic {G}eometry 1: {S}chemes. {W}ith {E}xamples and
  {E}xercises}.
\newblock Vieweg + Teubner Verlag, 2010.

\bibitem[Ill96]{luc}
Luc Illusie.
\newblock {F}robenius et d\'{e}g\'{e}n\'{e}rescence de {H}odge.
\newblock In {\em {I}ntroduction \`{a} la th\'{e}orie de {H}odge}, volume~3 of
  {\em Panor. Synth\`{e}ses}, pages 113--168. Soc. Math. France, Paris, 1996.

\bibitem[Kat72]{katz}
Nicholas~M. Katz.
\newblock {A}lgebraic solutions of differential equations (p-curvature and the
  {H}odge filtration).
\newblock {\em Invent. Math.}, 18:1--118, 1972.

\bibitem[Kat79]{katzsl}
Nicholas~M. Katz.
\newblock Slope filtration of {$F$}-crystals.
\newblock In {\em Journées de {G}éométrie {A}lgébrique de {R}ennes, {V}ol.
  {{\rm I}}}, volume~63 of {\em Astérisque}, pages 113--163, Paris, 1979. Soc.
  Math. France.

\bibitem[Lan83]{lang}
William~E. Lang.
\newblock {O}n {E}nriques surfaces in {C}haracteristic p. {I}.
\newblock {\em Math. Ann.}, 265(1):45--65, 1983.

\bibitem[Lau10]{lau}
Eike Lau.
\newblock {S}moothness of the truncated display functor.
\newblock {\em preprint}, 2010.
\newblock \url{http://arxiv.org/abs/1006.2723}.

\bibitem[LMB91]{laumon}
G\'{e}rard Laumon and Laurent Moret-Bailly.
\newblock {\em {C}hamps alg\'{e}briques}.
\newblock Springer-Verlag, 1991.

\bibitem[Mum70]{mumf}
David Mumford.
\newblock {\em {A}belian varieties}.
\newblock Tata Institute of Fundamental research, India, 1970.

\bibitem[MW04]{ben}
Ben Moonen and Torsten Wedhorn.
\newblock {D}iscrete invariants of varieties in positive characteristic.
\newblock {\em Int. Math. Res. Notices}, (72):3855--3903, 2004.

\bibitem[NVW10]{vasiu}
Marc–Hubert Nicole, Adrian Vasiu, and Torsten Wedhorn.
\newblock {P}urity of level m stratifications.
\newblock {\em Ann. Sci. Ec. Norm. Sup.}, 43(6):925--955, 2010.

\bibitem[Oor02]{repurity}
Frans Oort.
\newblock Purity reconsidered, 2002.
\newblock \url{http://www.staff.science.uu.nl/~oort0109/Purrec.ps}.

\bibitem[PWZ11]{paul}
Richard Pink, Torsten Wedhorn, and Paul Ziegler.
\newblock {A}lgebraic zip data.
\newblock {\em Documenta Math.}, (16):253--300, 2011.

\bibitem[PWZ12]{fzips}
Richard Pink, Torsten Wedhorn, and Paul Ziegler.
\newblock {F}-zips with additional structure.
\newblock {\em preprint}, 2012.
\newblock \url{http://arxiv.org/abs/1208.3547}.

\bibitem[Riz06]{rizov}
Jordan Rizov.
\newblock {M}oduli stacks of polarized {K}3 surfaces in mixed characteristic.
\newblock {\em Serdica Math. J.}, 32(2-3):131--178, 2006.
\newblock \url{http://arxiv.org/abs/math/0506120v2}.

\bibitem[Spr98]{springer}
Tonny~Albert Springer.
\newblock {\em {{L}inear {A}lgebraic {G}roups}}.
\newblock Birkhäuser, 2 edition, 1998.

\bibitem[SR05]{santos}
Walter~Ferrer Santos and Alvaro Rittatore.
\newblock {\em {A}ctions and {I}nvariants of {A}lgebraic {G}roups}.
\newblock Chapman {\&} Hall/CRC, 2005.

\bibitem[Vas02]{vasiu2}
Adrian Vasiu.
\newblock {C}rystalline boundedness principle.
\newblock {\em Ann. Sci. Ec. Norm. Sup.}, 39(4):245--300, 2002.

\bibitem[Vie11]{vieh}
Eva Viehmann.
\newblock {T}runcations of level 1 of elements in the loop group of a reductive
  group.
\newblock {\em preprint}, 2011.
\newblock \url{http://arxiv.org/abs/0907.2331}.

\bibitem[VW12]{eva}
Eva Viehmann and Torsten Wedhorn.
\newblock {E}kedahl-{O}ort and {N}ewton strata for {S}himura varieties of
  {P}{E}{L} type.
\newblock {\em preprint}, 2012.
\newblock \url{http://arxiv.org/abs/1011.3230}.

\bibitem[Wed01]{oortstrata}
Torsten Wedhorn.
\newblock {\em {T}he dimension of {O}ort strata of {S}himura varieties of
  {PEL}-type}, volume 195 of {\em Progr. Math}.
\newblock Birkhäuser, Basel, 2001.
\newblock \url{http://arxiv.org/abs/0808.1629}.

\bibitem[Wed08]{torsten}
Torsten Wedhorn.
\newblock {D}e {R}ham {C}ohomology of varieties over fields of positive
  characteristic.
\newblock {\em Higher-dimensional geometry over finite fields}, pages 269--314,
  2008.

\end{thebibliography}
\end{document}